\documentclass{article}
\usepackage{amsmath,amssymb,amsthm,graphicx,pifont,geometry,fleqn,txfonts,sidecap,subfig,caption,color}
\usepackage{graphicx}
\graphicspath{%
    {saved_images/}
}
\begin{document}

{
\theoremstyle{definition}
\newtheorem{examp}{Example}
\newtheorem*{defi}{Definition}
\newtheorem*{notation}{Notation}
}

\newcommand\set[1]{\{#1\}}
\def\minrec{MINREC}
\def\minfearset{MINFAS}
\def\conf{configuration}
\def\confs{configurations}
\def\reverse{reverse}
\newcommand\outdegree[2]{\deg_{#1}^{+}(#2)}
\newcommand\indegree[2]{\deg_{#1}^{-}(#2)}
\newcommand\degree[2]{{\deg}_{#1}(#2)}
\newcommand\directize[1]{\overset{\leftrightarrow}{#1}}
\newcommand\deleteoutarcs[2]{#1_{\backslash #2^{+}}}
\newcommand\confdegree[3]{\textbf{\texttt{sum}}\,_{#1,#2}(#3)}
\newcommand\stab[3]{{#3}^{\circ^{#2}}}
\newcommand\sinkaug[3]{\overline{{#1}^{*^{#2}}}^{#3}}
\newcommand\augnum[2]{\mathcal{I}_{\!#1}(#2)}
\newcommand\seaugnum[1]{\mathcal{I}'(#1)}
\newcommand\numarcs[3]{\deg_{#1}(#2,#3)}
\newcommand\level[2]{level_{#1}(#2)}
\newcommand\Tut[2]{\mathcal{T}_{#1}(#2)}
\newcommand\contract[2]{#1_{/#2}}
\newcommand\tail[1]{#1^{-}}
\newcommand\head[1]{#1^{+}}
\newcommand\deletearc[2]{#1_{\backslash #2}}
\newcommand\numloops[1]{L(#1)}
\newcommand\deleteloops[1]{\overline{#1}}
\newcommand\totalchipmin[1]{\kappa(\overline{#1})}
\newcommand\vertexcontract[2]{#1_{/#2}}
\newcommand\support[1]{N_F(#1)}
\newtheorem{lem}{Lemma}
\newtheorem{theo}{Theorem}
\newtheorem*{ques}{Question}
\newtheorem{prop}{Proposition}
\newtheorem{prob}{Problem}
\newtheorem{conj}{Conjecture}

\title{Chip-firing game and partial Tutte polynomial for\\Eulerian digraphs\thanks{This paper was partially sponsored by Vietnam Institute for Advanced Study in Mathematics (VIASM) and the Vietnamese National Foundation for Science and Technology Development (NAFOSTED)}}

\author{K\'evin Perrot and Trung Van Pham}

\date{\today}

\maketitle
\begin{abstract}
The Chip-firing game is a discrete dynamical system played on a graph, in which chips move along edges according to a simple local rule. Properties of the underlying graph are of course useful to the understanding of the game, but since a conjecture of Biggs that was proved by Merino L\'opez, we also know that the study of the Chip-firing game can give insights on the graph. In particular, a strong relation between the partial Tutte polynomial $T_G(1,y)$ and the set of recurrent configurations of a Chip-firing game (with a distinguished sink vertex) has been established for undirected graphs. A direct consequence is that the generating function of the set of recurrent configurations is independent of the choice of the sink for the game, as it characterizes the underlying graph itself. In this paper we prove that this property also holds for Eulerian directed graphs (digraphs), a class on the way from undirected graphs to general digraphs. It turns out from this property that the generating function of the set of recurrent configurations of an Eulerian digraph is a natural and convincing candidate for generalizing the partial Tutte polynomial $T_G(1,y)$ to this class. Our work also gives some promising directions of looking for a generalization of the Tutte polynomial to general digraphs.\\
\text{}\\
\textbf{Keywords.} Chip-firing game, complexity, critical configuration, Eulerian digraph, feedback arc set, recurrent configuration, reliability polynomial, Sandpile model, Tutte polynomial.
\end{abstract}

\section{Introduction}

There are insightful polynomials that are defined on undirected graphs, such as Tutte polynomial, chromatic polynomial, cover polynomial, reliability polynomial, \emph{etc}, which evaluations count certain combinatorial objects. The Tutte polynomial is the most well-known, it has many interesting properties and applications \cite{Tut53}. There is an evident interest in looking for analogues of the Tutte polynomial for directed graphs (digraphs), of for some other objects \cite{Ges89,Gor93, CG95}. These attempts share properties of the Tutte polynomial. Nevertheless, they are not natural extensions of the Tutte polynomial in the sense that one does not know a conversion from the properties of these polynomials to those of the Tutte polynomial, in particular how to get back to the Tutte polynomial on undirected graphs from these polynomials. For this reason the authors of \cite{CG95} asked for a natural generalization of Tutte polynomial for digraphs.

The evaluation of the Tutte polynomial $T_G(x,y)$ at $x=1$ is important since it has a strong connection to the reliability polynomial that is studied in the network theory. In this paper we present a polynomial that can be considered as a natural generalization of $T_G(1,y)$ for the class of Eulerian digraphs. An Eulerian digraph is a strongly connected digraph in which each vertex has equal in-degree and out-degree. An undirected graph can be regarded as an Eulerian digraph by replacing each edge $e$ by two {\reverse} arcs $e'$ and $e''$ that have the same endpoints as $e$. When considering undirected graphs seen as Eulerian digraphs in that way, we will see that we get back to the partial Tutte polynomial $T_G(1,y)$, which is a new and relevant feature. 

This work is based on an idea conjectured by Biggs and proved by Merino L\'opes, that the generating function of the set of recurrent configurations of the Chip-firing game of an undirected graph is equal to the partial Tutte polynomial $T_G(1,y)$ \cite{Big97,Lop97}. Based on a discrete dynamical system, this construction defines a polynomial that characterizes the graph supporting the dynamic. It is not straightforward to generalize those ideas to the class of Eulerian digraphs, but the results we will develop gives a promising direction for further extensions.

The Chip-firing game is a discrete dynamical system defined on a directed graph (digraph) $G$, where some chips are stored on each vertex of $G$. An assignment of chips on the vertices is called a \emph{configuration} of $G$, and a configuration can be transformed into a new configuration by the following rule: if a vertex $v$ of $G$ has as many chips as its out-degree and at least one out-going arc, then it is \emph{firable} and the diffusion process called \emph{firing v} consists in moving one chip of $v$ along each out-going arc to the corresponding vertex. The game playing with this rule is called \emph{Chip-firing game} (CFG), and $G$ is called \emph{support graph} of the game. A configuration is \emph{stable} if it has no firable vertex. It is known that starting from any initial configuration the game either plays forever or converges to a unique stable configuration. If $G$ has a \emph{global sink}, \emph{i.e.}, a vertex $s$ with out-degree $0$ and such that for any other vertex $v$ there is a path from $v$ to $s$, then the game always converges for any choice of initial configuration \cite{BLS91,BL92,HLMPPW08}. Throughout the paper we will concentrate on such CFG with a global sink. For a strongly connected digraph $G$, we will choose a particular vertex $s$ and consider it as the sink by removing all out-going arcs of $s$. The study of sink-independent properties (definitions that leads to the same object whatever vertex is chosen as the sink) will provide clues to define a natural analogue of the Tutte polynomial, for the class of Eulerian digraphs. The Chip-firing with a sink on digraphs has been introduced under the name Dollar game on undirected graphs.

The Dollar game is a variant of CFG on undirected graphs in which a particular vertex plays the role of a sink, and the sink can only be fired if all other vertices are not firable \cite{Big99}. In this model the number of chips stored in the sink may be negative. This definition leads naturally to the notion of \emph{recurrent configurations} (originally called \emph{critical configurations}) that are stable, and unchanged under firing at the sink and stabilizing the resulting configuration. The definition of the Dollar game on Eulerian digraphs is the same as on undirected graphs, \emph{i.e.} some vertex is chosen to be the sink that only can be fired only if all other vertices are not firable \cite{HLMPPW08}. In the rest of the paper we will use the name Chip-firing game with a sink instead of Dollar game.

The set of recurrent configurations of a CFG with a sink on an undirected graph has many interesting properties, such as it is an Abelian group with the addition defined by vertex to vertex addition of chip content followed by stabilization, and its cardinality is equal to the number of spanning trees of the support graph, \emph{etc}. Remarkably, Biggs defined the level of a recurrent configuration and made an intriguing conjecture about the relation between the generating function of recurrent configurations and the Tutte polynomial \cite{Big97}. This conjecture was later proved by Merino L\'opez \cite{Lop97}. A direct consequence is that the generating function of recurrent configurations of a CFG with a sink is independent of the chosen sink, and thus characterizes the support graph. This fact is definitely not trivial, and opened a new direction for studying graphs using the Chip-firing game as a tool \cite{CB03,Mer05}.

A lot of properties of recurrent configurations on undirected graphs can be extended to Eulerian digraphs without difficulty. However the situation is different when one tries to extend the sink-independence property of the generating function to a larger class of graphs, in particular to Eulerian digraphs, mainly because a natural definition of the Tutte polynomial for digraphs is unknown, nor is it for Eulerian digraphs. In this paper we develop a combinatorial approach, based on a level-preserving bijection between two sets of recurrent configurations with respect to two different sinks, to show that this sink-independence property also holds for Eulerian digraphs. This bijection provides new insights into the groups of recurrent configurations.

It turns out from the sink-independence property of the generating function, that this latter is a characteristic of the support Eulerian digraph, and we can denote it by $\Tut{G}{y}$ regardless of the sink. We will see that evaluations of $\Tut{G}{y}$ can be considered as extensions of $T_G(1,y)$ to Eulerian digraphs, which make us believe that the polynomial $\Tut{G}{y}$ is a natural generalization of $T_G(1,y)$. Furthermore, the most important feature is that $\Tut{G}{y}$ and $T_G(1,y)$ are equal on undirected graphs. It requires to be inventive to discover which objects the evaluations of $\Tut{G}{y}$ counts, and we hope that further properties will be found. The class of Eulerian digraphs is in-between undirected and directed graph, and following the track we develop in this paper, we propose some conjectures that would be promising directions of looking for a natural generalization of $T_G(x,y)$ to general digraphs.


The paper is divided into the following sections. Section \ref{recurrent section} recalls known results on recurrent configurations on a digraph with global sink. Section \ref{Eulerian section} is devoted to the Eulerian digraph case, and Section \ref{sink-independence section} establishes the sink-independence of the generating function of recurrent configurations in that case. The Tutte polynomial generalization is presented in Section \ref{tutte-like section}, and Section \ref{open section} hints at continuations of the present work.


\section{Recurrent configurations on a digraph with global sink}
\label{recurrent section}

All graphs in this paper are assumed to be multi-digraphs without loops. Graphs with loops will be considered in Section \ref{tutte-like section}. We introduce in this section some notations and known results about recurrent configurations of CFG with a sink on general digraphs, followed by straightforward considerations on the number of chips stored on vertices of recurrent configurations.

For a digraph $G=(V,A)$ and an arc $e \in A$, we denote by $\tail{e}$ and $\head{e}$ the tail and head of $e$, respectively. For two vertices $v,v'\in V$, let $\numarcs{G}{v}{v'}$ denote the number of arcs from $v$ to $v'$ in $G$. A configuration $c$ on $G$ is a map from $V$ to $\mathbb{N}$. A vertex $v$ is \emph{firable} in $c$ if and only if $c(v)\geq \outdegree{G}{v}>0$. Firing a firable vertex $v$ is the process that decreases $c(v)$ by $\outdegree{G}{v}$ and increases each $c(v')$ with $v'\neq v$ by $\numarcs{G}{v}{v'}$. A sequence $(v_1,v_2,\cdots,v_k)$ of vertices of $G$ is called a \emph{firing sequence} of a configuration $c$ if starting from $c$ we can consecutively fire the vertices $v_1,v_2,\cdots,v_k$. Applying the firing sequence leads to configuration $c'$ and we write $c\overset{v_1,v_2,\cdots,v_k}{\longrightarrow}c'$, or $c\overset{*}{\to}c'$ without specifying the firing sequence.


In the rest of this section we assume that $G$ has a global sink $s$. The definition of recurrent configurations is based on the convergence of the game, which is ensured if $G$ has a global sink. Since $s$ is not firable no matter how many chips it has, it makes sense to define a configuration to be a map from $V\backslash\set{s}$ to $\mathbb{N}$. When a chip goes into the sink, it vanishes. The interest is to assimilate two configurations that have the same number of chips on every vertices except on the sink. Note that in this section we consider only one fixed sink, but in subsequent sections we will consider the CFG relatively to different choices of sink, and therefore we will need some more notations. Let us not be overburdened yet, a configuration on $G$ with sink $s$ is a map $V\backslash\set{s} \to \mathbb{N}$.

We recall a basic result of the Chip-firing game on digraphs with a global sink.

\begin{lem}\cite{BLS91,BL92,HLMPPW08}
\label{CFG:convergence}
For any initial configuration $c$ the game converges to a unique stable configuration, denoted by $c^{\circ}$. Moreover, let $\mathfrak{f}$ and $\mathfrak{f}^{'}$ be two firing sequences of $c$ such that $c\overset{\mathfrak{f}}{\to}c^{\circ}$ and $c\overset{\mathfrak{f}'}{\to}c^{\circ}$, then for every vertex $v \neq s$ the number of times $v$ occurs in $\mathfrak{f}$ is the same as in $\mathfrak{f}'$.
\end{lem}

The following is simple but very important, and will often be used without explicit reference.

\begin{lem}
For two configurations $c$ and $d$, we denote by $c+d$ the configuration given by $(c+d)(v)=c(v)+d(v)$ for any $v \neq s$. Then $(c+d)^{\circ}=(c^{\circ}+d)^{\circ}$.
\end{lem}

\begin{defi}
A stable configuration $c$ is \emph{recurrent} if and only if for any configuration $d$ there is a configuration $d'$ such that $c=(d+d')^{\circ}$.
\end{defi}
There are several equivalent definitions of \emph{recurrent} configurations. The one above says that $c$ is recurrent if and only if it can be reached from any other configuration $d$ by adding some chips (according to $d'$) and then stabilize.

Dhar proved that the set of recurrent configurations has an elegant algebraic structure \cite{Dha90}. Fix a linear order $v_1\prec v_2\prec \cdots \prec v_{n-1}$ on the vertices different from $s$, where $n=|V|$. Now a configuration of $G$ can be represented as a vector in $\mathbb{Z}^{n-1}$. For each $i \in [1..n-1]$ let $r_i$ be the vector in $\mathbb{Z}^{n-1}$ defined by $r_{i,j}=\numarcs{G}{v_i}{v_j}$ if $i\neq j$, otherwise $r_{i,j}=-\outdegree{G}{v_i}$ if $i=j$. Firing index $i$ then corresponds to adding the vector $r_i$. We define a binary relation $\sim$ over $\mathbb{Z}^{n-1}$ by $d\sim d'$ iff there exist $a_1,a_2,\cdots,a_{n-1}\in \mathbb{Z}$ such that $d-d'=\underset{1 \leq i\leq n-1}{\sum} a_i r_i$, \emph{i.e.} $d$ and $d'$ are linked by a (possibly impossible to perform) sequence of firings. The following states the nice algebraic structure of the set of all recurrent configurations of $G$ with sink $s$.

\begin{lem}\cite{HLMPPW08}
\label{CFG:recurrent:property:algebraic}
The set of all recurrent configurations of $G$ is an Abelian group with the addition $\oplus$ defined by $c_1\oplus c_2:=(c_1+c_2)^{\circ}$. This group is isomorphic to $\mathbb{Z}^n/\!\!<\!\!r_1,r_2,\cdots,r_{n-1}\!\!>$. Moreover, each equivalence class of $\mathbb{Z}^{n-1}/\!\!\sim$ contains exactly one recurrent configuration, and the number of recurrent configurations is equal to the number of equivalence classes.
\end{lem}

The group in Lemma \ref{CFG:recurrent:property:algebraic} is called the \emph{Sandpile group} of $G$. The following simple properties can be derived easily from the definition of recurrent configuration.

\begin{lem}
\label{CFG:stabilization of large configuration}
The following holds
\begin{itemize}
\item[1.] Let $c$ be a configuration such that $c(v)\geq \outdegree{G}{v}-1$ for every $v\neq s$. Then $c^{\circ}$ is recurrent.
\item[2.] Let $c$ and $c'$ be two configurations such that $c(v)\leq c'(v)$ for any $v \neq s$. Then $\underset{v\neq s}{\sum}c(v)-\underset{v\neq s}{\sum}c^{\circ}(v)\leq \underset{v\neq s}{\sum}c'(v)-\underset{v\neq s}{\sum} {c'}^{\circ}(v)$. Moreover, if $c^{\circ}$ is recurrent then $c'^{\circ}$ is also recurrent.
\end{itemize}
\end{lem}

\begin{proof}
\text{}\\[-1em]
\begin{itemize}
\item[1.] For any configuration $d$, adding grains according to $d'=c-d^{\circ}$ leads to $c$, and $d'$ is a configuration with positive chip content on each vertex. Clearly, $(d+d')^{\circ}=(d+c-d^{\circ})^{\circ}=(d^{\circ}+c-d^{\circ})^{\circ}=c^{\circ}$, therefore $c^{\circ}$ is recurrent.
\item[2.] Let $\mathfrak{f}=(v_1,v_2,\cdots,v_k)$ be a firing sequence of $c$  such that $c\overset{\mathfrak{f}}{\to}c^{\circ}$. Since $\underset{v\neq s}{\sum}c(v)-\underset{v\neq s}{\sum}c^{\circ}(v)$ is the number of chips lost into the sink, we have $\underset{v\neq s}{\sum}c(v)-\underset{v\neq s}{\sum}c^{\circ}(v)=\underset{1 \leq i\leq k}{\sum} \numarcs{G}{v_i}{s}$. Since $c(v)\leq c'(v)$ for any $v\neq s$, $\mathfrak{f}$ is also a firing sequence of $c'$. Therefore there is a firing sequence $\mathfrak{f}'=(v_1,v_2,\cdots,v_k,v_{k+1},v_{k+2},\cdots,v_l)$ of $c'$ such that $c'\overset{\mathfrak{f}'}{\to}{c'}^{\circ}$. For the same reason we have $\underset{v\neq s}{\sum} c'(v)-\underset{v\neq s}{\sum} {c'}^{\circ}(v)=\underset{1 \leq i \leq l}{\sum}d(v_i,s)$. The first claim follows.

Let $d$ be an arbitrary configuration. Since $c^{\circ}$ is recurrent, there is a configuration $d'$ such that $(d+d')^{\circ}=c^{\circ}$. Let $d''=d'+c'-c$ be a configuration. We have $(d+d'')^{\circ}=(d+d'+c'-c)^{\circ}=((d+d')^{\circ}+c'-c)^{\circ}=(c^{\circ}+c'-c)^{\circ}=(c+c'-c)^{\circ}=c'^{\circ}$, thus $c'^{\circ}$ is recurrent.\\[-2em]
\end{itemize}
\end{proof}

\section{Chip-firing game on an Eulerian digraph with a sink}
\label{Eulerian section}

Let $G=(V,A)$ be a digraph. The digraph $G$ is \emph{Eulerian} if $G$ is connected and for every $v \in V$ we have $\indegree{G}{v}=\outdegree{G}{v}$. With this condition the digraph $G$ is strongly connected. In this section we assume that $G$ is Eulerian, and present properties that recurrent configurations verify in that case.

As in the previous section, the definition of recurrent configuration is based on the convergence of the game. Therefore a global sink plays an important role in the definition. The digraph $G$ is strongly connected, therefore it has no global sink and the game may play forever from some initial configurations. To overcome this issue, we distinguish a particular vertex of $G$ that plays the role of the sink. Let $s$ be a vertex of $G$, by removing all outgoing arcs of $s$ from $G$ we got the digraph $\deleteoutarcs{G}{s}$ that has a global sink $s$. The Chip-firing game on $G$ with sink $s$ is the ordinary Chip-firing game that is defined on $\deleteoutarcs{G}{s}$, and recurrent configurations are defined as presented above, on $\deleteoutarcs{G}{s}$. Figures \ref{fig:im1} and \ref{fig:im2} present an example of $G$ and $\deleteoutarcs{G}{s}$. It is a good way to think of the Chip-firing game on an Eulerian digraph with a sink as the ordinary Chip-firing game on $G$ with a fixed vertex that never fires in the game no matter how many chips it has. In this section we consider a fixed sink $s$.

A configuration of the Chip-firing game on $G$ with sink $s$ is a map from $V(G)\backslash \set{s}$ to $\mathbb{N}$. To verify the recurrence of a configuration $c$, we have to test the condition that for any configuration $d$ there is a configuration $d'$ such that $(d+d')^{\circ}=c$. This is a tiresome task. However, in the case of Eulerian digraphs we have the following useful criterion.

\begin{lem}\cite{Dha90,HLMPPW08}
\label{CFG:Dah}
A configuration $c$ is recurrent  if and only if $(c+\beta)^{\circ}=c$, where $\beta$ is the configuration defined by $\beta(v)=\numarcs{G}{s}{v}$ for every $v \neq s$. Moreover, if $c$ is recurrent then each vertex distinct from $s$ occurs exactly once in any firing sequence $\mathfrak{f}$ from $c+\beta$ to $(c+\beta)^{\circ}$.
\end{lem}

\begin{figure}
\centering
\subfloat[$G$]{\label{fig:im1}\includegraphics[bb=0 0 260 118,width=1.88in,height=0.848in,keepaspectratio]{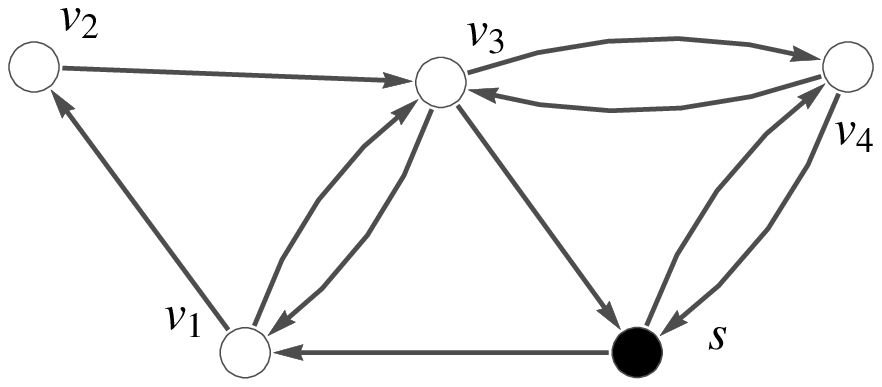}}\quad \quad
\subfloat[$\deleteoutarcs{G}{s}$]{\label{fig:im2}\includegraphics[bb=0 0 260 118,width=1.88in,height=0.848in,keepaspectratio]{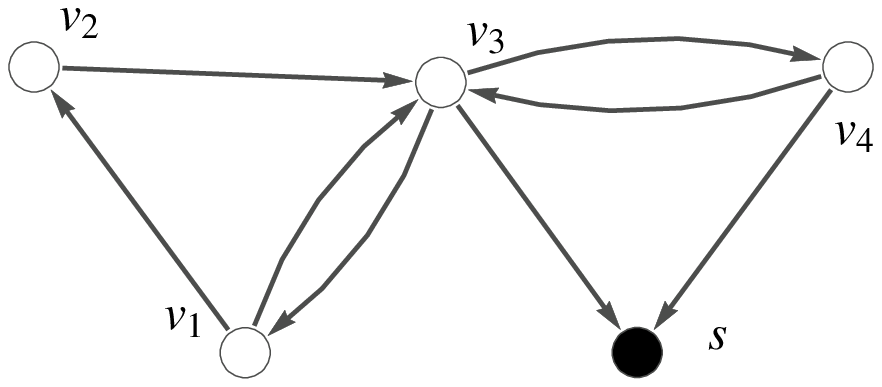}}\\
\subfloat[A configuration $c$]{\label{fig:im3}\includegraphics[bb=0 0 352 157,width=1.88in,height=0.839in,keepaspectratio]{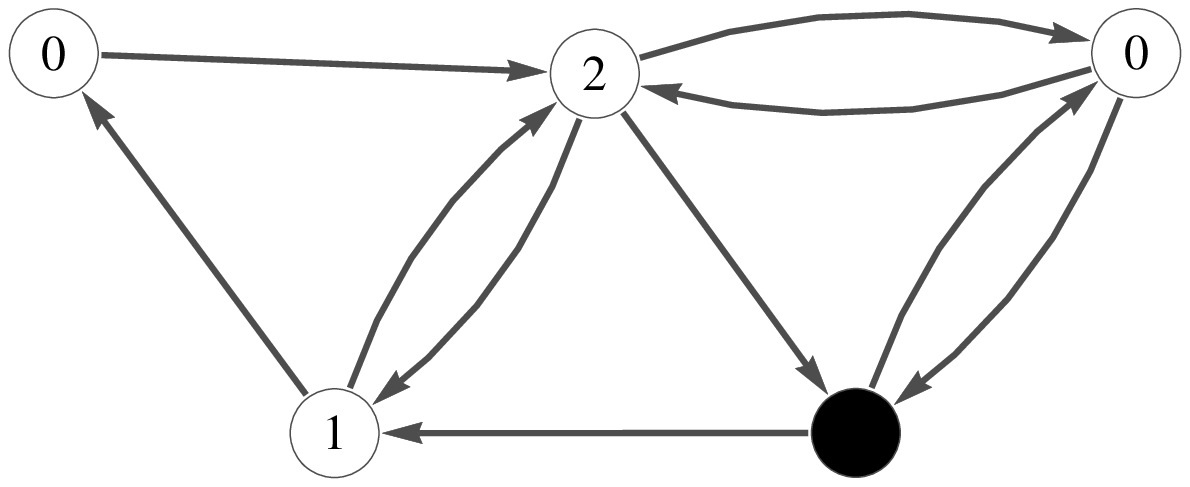}}\quad \quad
\subfloat[$c+\beta$]{\label{fig:im4}\includegraphics[bb=0 0 330 148,width=1.88in,height=0.841in,keepaspectratio]{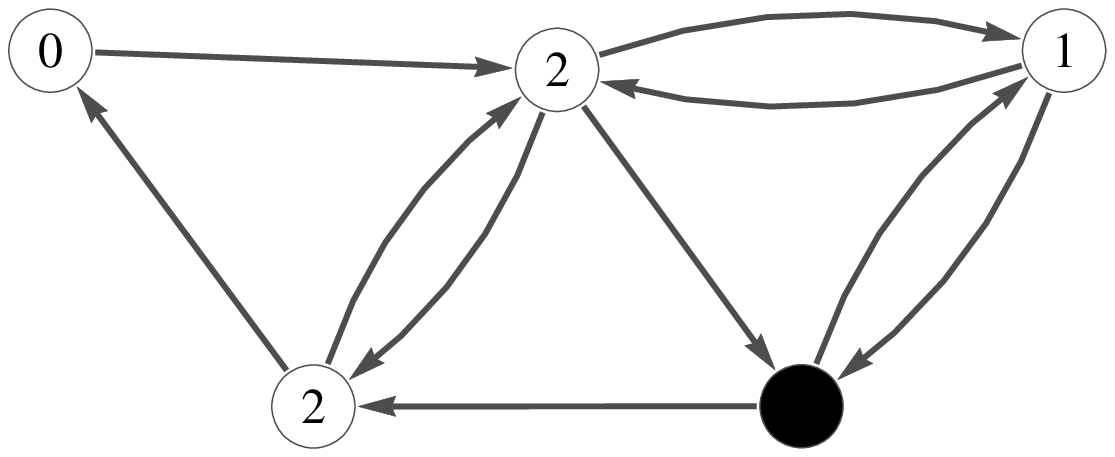}}
\caption{Burning algorithm}
\label{fig:im01020304}
\end{figure}

Figure \ref{fig:im3} presents a configuration $c$. The configuration $c+\beta$ is presented in Figure \ref{fig:im4}, adding $\beta$ corresponds to firing the sink. To verify the recurrence of $c$, one computes $(c+\beta)^{\circ}$. Starting with $c+\beta$ we fire consecutively the vertices $v_1,v_3,v_2,v_4$ in this order and get exactly the configuration $c$, therefore $c$ is recurrent. This procedure is called \emph{Burning algorithm}. The following will be important later.

\begin{lem}
\label{CFG:maximum of total number of chips}
Let $c$ and $c'$ be two stable configurations such that $c$ and $c'$ are in the same equivalence class. If $c$ is recurrent then $\underset{v\neq s}{\sum}c'(v)\leq \underset{v\neq s}{\sum} c(v)$. 
\end{lem}

\begin{proof}
Let $\beta$ be the configuration that is defined as in Lemma \ref{CFG:Dah} and $d$ be an arbitrary stable configuration. We claim that for any firing sequence $\mathfrak{f}=(v_1,v_2,\cdots,v_k)$ of $d+\beta$ such that $d+\beta\overset{\mathfrak{f}}{\to}(d+\beta)^{\circ}$ each vertex of $G$ occurs at most once in $\mathfrak{f}$. For a contradiction we assume otherwise. This assumption implies that there is a first repetition, \emph{i.e.}, there is $p\in [1..k]$ such that $v_1,v_2,\cdots,v_{p-1}$ are pairwise-distinct and $v_p=v_q$ for some $q \in [1..p-1]$. We denote $d'$ the configuration obtained from $d+\beta$ after the vertices $v_1,v_2,\cdots,v_{p-1}$ have been fired. We will now show that $v_p$ is not firable in $d'$, a contradiction. Let $r$ be the number of chips $v_p$ receives from its in-neighbors when the vertices $v_1,v_2,\cdots,v_{p-1}$ have been fired. Adding $\beta$ corresponds to firing the sink, thus $d'(v_p)=d(v_p)+\numarcs{G}{s}{v_p}+r - \outdegree{G}{v_p}$. Since $v_1,v_2,\cdots,v_{p-1}$ are pairwise-distinct and different from the sink, we have $r+\numarcs{G}{s}{v_p} \leq \indegree{G}{v_p}$. The digraph $G$ is Eulerian, therefore $\indegree{G}{v_p} = \outdegree{G}{v_p}$ and from the previous equality we have $d'(v_p) \leq d(v_p)$, but $d$ is stable so vertex $v_p$ is not firable in configuration $d'$, which is absurd.

Since each of the in-neighbors of $s$ is fired at most once in any firing sequence $\mathfrak{f}$ of $d+\beta$ such that $d+\beta\overset{\mathfrak{f}}{\to}(d+\beta)^{\circ}$, it follows that no more chips than that added to $d$ (that is $\sum \limits_{v \neq s} \beta(v)$) can end up in the sink since $G$ is Eulerian, and consequently $\underset{v \neq s}{\sum} d(v)\leq \underset{v \neq s}{\sum} (d+\beta)^{\circ}(v)$. Repeating the application of this inequality $n$ times we have $\underset{v\neq s}{\sum} d(v)\leq \underset{v\neq s}{\sum} (d+n\beta)^{\circ}(v)$, where $n\beta$ is the configuration given by $(n\beta)(v)=n\,\beta(v)$ for any $v\neq s$. This reasoning can be applied to $c'$ and we have $\underset{v\neq s}{\sum}c'(v)\leq \underset{v\neq s}{\sum} (c'+n\beta)^{\circ}(v)$ for any $n \in \mathbb{N}$. Since for any vertex $v\neq s$ and any $v'$ being an out-neighbor of $s$ there is a path in $\deleteoutarcs{G}{s}$ from $v'$ to $v$, with $n$ large enough we can add a sufficient number of chips so that there is an appropriate firing sequence $\mathfrak{f}'$ of $c'+n\beta$ with $c'+n \beta\overset{\mathfrak{f}'}{\to}c''$ and such that $c''(v)\geq \outdegree{G}{v}-1$ for any $v\neq s$. When stabilizing $c''$, it follows from Lemma \ref{CFG:convergence} (convergence), Lemma \ref{CFG:stabilization of large configuration} (recurrence) and Lemma \ref{CFG:recurrent:property:algebraic} (unicity of recurrent configuration in an equivalent class) that it leads to $c$, that is, $c=c''^{\circ}=(c'+n\beta)^{\circ}$. This completes the proof.
\end{proof}

\begin{ques}
Does the claim of Lemma \ref{CFG:maximum of total number of chips} hold for a general digraph with a global sink?
\end{ques}

Note that if this statement is true, then it is tight. Figure \ref{fig:conf12} presents an example, on an undirected graph, of a recurrent configuration and a non-recurrent configuration belonging to the same equivalence class, such that they contain the same total number of chips. As a consequence, the recurrent configuration is not necessarily the unique configuration of maximum total number of chips over stable configurations of its equivalence class.

\begin{figure}
\centering
\subfloat[A recurrent configuration]{\label{fig:conf1}\includegraphics[bb=0 0 273 90,width=1.96in,height=1.22in]{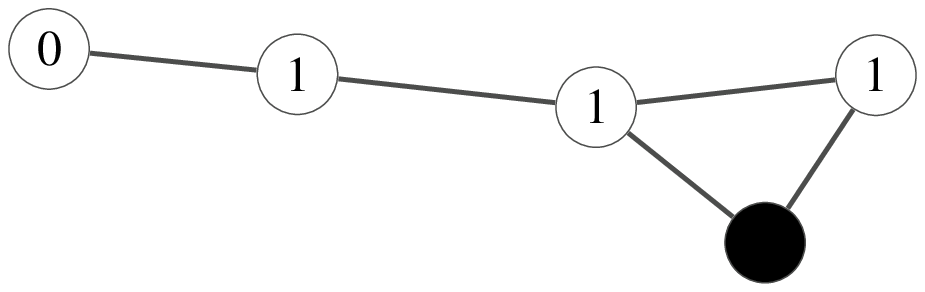}}\quad \quad
\subfloat[A non-recurrent configuration]{\label{fig:conf2}\includegraphics[bb=0 0 278 92,width=1.96in,height=1.22in]{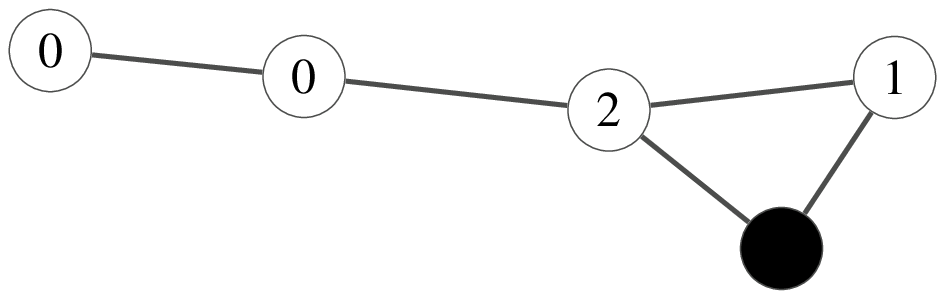}}
\caption{Two stable configurations from the same equivalence class on an undirected graph}
\label{fig:conf12}
\end{figure}

\section{Sink-independence of generating function of recurrent configurations of an Eulerian digraph}
\label{sink-independence section}

\begin{figure}[!h]
\centering
\includegraphics[bb=0 0 736 224,width=4.34in,height=1.32in,keepaspectratio]{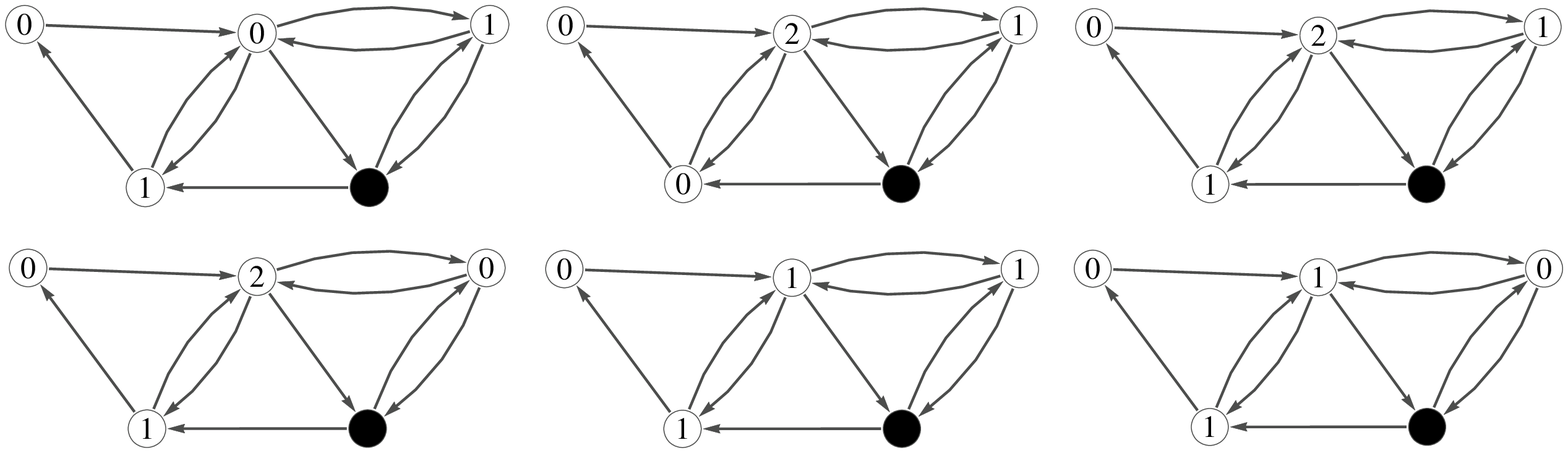}
\caption{Recurrent configurations with respect to sink $s$}
\label{fig:im5}
\end{figure}

\begin{figure}[!h]
\centering
\includegraphics[bb=0 0 790 247,width=4.34in,height=1.36in,keepaspectratio]{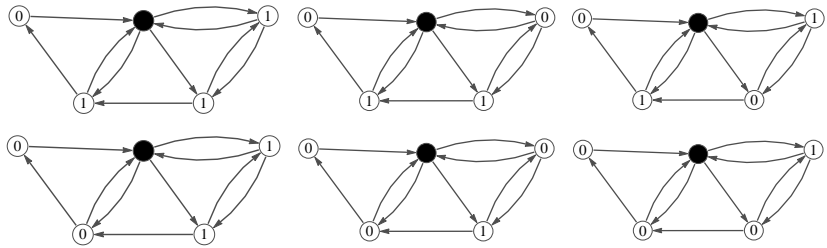}
\caption{Recurrent configurations with respect to sink $v_3$}
\label{fig:im6}
\end{figure}

\hspace{4.ex} \textbf{Key observation.}
Let us give an important observation that motivates the study presented in this paper. We consider the Chip-firing game on the digraph drawn on Figure \ref{fig:im1}. In this game the vertex $s$ is chosen to be sink. All the recurrent configurations are presented in Figure \ref{fig:im5}. For each recurrent configuration we compute the sum of chips on the vertices different from the sink. We get the sorted sequence of numbers $(2, 2, 3, 3, 3, 4)$. If $v_3$ is chosen to be the sink of the game, all the recurrent configurations are given in Figure \ref{fig:im6}, and the sum of chips on vertices different from the sink gives the sorted sequence $(1, 1, 2, 2, 2, 3)$. The two sequences are the same up to adding a constant sequence. This property also holds with other choices of sink, therefore, up to a constant, this sequence is characteristic of the support graph itself. This interesting property is the main discovery exploited in this paper, and allows to generalize the construction presented in \cite{Big97} and proved in \cite{Lop97} of an analogue for the Tutte polynomial to the class of Eulerian digraphs. It is stated in the following theorem.

\begin{theo}
\label{CFG:sink-independence}
Let $G$ be an Eulerian digraph and $s$ a vertex of $G$. For each recurrent configuration with respect to sink $s$, let $\confdegree{G}{s}{c}$ denote $\outdegree{G}{s}+\underset{v \neq s}{\sum}c(v)$. The recurrent configurations with respect to sink $s$ are denoted by $c_1,c_2,\cdots,c_p$ for some $p$. Then the sequence $(\confdegree{G}{s}{c_i})_{1\leq i \leq p}$ is independent of the choice of $s$ up to a permutation of the entries.
\end{theo}

The result of Merino L\'opez \cite{Lop97} implies that Theorem \ref{CFG:sink-independence} is true for undirected graphs. An undirected graph $G$ can be considered as an Eulerian digraph by replacing each undirected edge $e$ with two endpoints $v$ and $v'$ by two {\reverse} directed arcs $e'$ and $e''$ satisfying $\tail{e'}=\head{e''}=v$ and $\head{e'}=\tail{e''}=v'$. With this conversion it makes sense to call an Eulerian digraph $G$ \emph{undirected} if for any two vertices $v,v'$ of $G$ we have $\numarcs{G}{v}{v'}=\numarcs{G}{v'}{v}$. The following known result is thus a particular case of Theorem \ref{CFG:sink-independence}, for the class of undirected graphs.

\begin{theo}\cite{Lop97}
Let $\mathcal{C}$ be the set of all recurrent configurations with respect to some sink $s$. If $G$ is an undirected graph (defined as a digraph) then $T_G(1,y)=\underset{c \in \mathcal{C}}{\sum} y^{level(c)}$, where $T_G(x,y)$ is the Tutte polynomial of $G$ and $level(c):=-\frac{|A|}{2}+\outdegree{G}{s}+\underset{v \neq s}{\sum}c(v)$ for any $c \in \mathcal{C}$.
\end{theo}

In the rest of this section we work with an Eulerian digraph $G=(V,A)$. In order to prove Theorem \ref{CFG:sink-independence}, we consider the following natural approach. Let $s_1$ and $s_2$ be two distinct vertices of $G$. We denote by $\mathcal{C}_1$ and $\mathcal{C}_2$ the sets of all recurrent configurations with respect to sink $s_1$ and $s_2$, respectively. We are going to construct a bijection $\theta$ from $\mathcal{C}_1$ to $\mathcal{C}_2$ such that $\confdegree{G}{s_1}{c}=\confdegree{G}{s_2}{\theta(c)}$ for every $c \in \mathcal{C}_1$. Note that it follows from \cite{HLMPPW08} that $|\mathcal{C}_1|=|\mathcal{C}_2|$. In order to work on the CFG with respect to different sinks, we introduce some clear notations.

\begin{notation}
  For a digraph $G=(V,E)$, let $c$ denote a configuration that assigns a number of chips to every vertex, {\em i.e.}, a map $c: V \to \mathbb N$. In order to discard the number of chips stored in the sink $s$, we introduce the notation $c^{*^{s}}$, which is the map $c$ restricted to the domain $V \backslash \set{s}$. We denote $c^{\circ^{s}}$ the configuration obtained by stabilizing with respect to the sink $s$. The operator ${\circ^{s}}$ can be applied to $c$ or $c^{*^{s}}$ and gives respectively a configuration $c^{\circ^{s}} : V \to \mathbb N$ or $c^{\circ^{s}*^{s}} : V\backslash \set{s} \to \mathbb N$ (note that $c^{\circ^{s}*^{s}}=(c^{*^{s}})^{\circ^{s}}$ since two operators $*^{s}$ and $\circ^{s}$ commute).
\end{notation}

It follows from Lemma \ref{CFG:convergence} that the configurations $c^{\circ^{s}}$ and $c^{\circ^{s}*^s}$ are well-defined and unique. See Figure \ref{fig:im0708} for an illustration. A basic trick will be to stabilize according to a sink $s_1$, choose a number of chips to assign to $s_1$, and then stabilize according to another sink $s_2$. Let us already state notations for this purpose.

\begin{notation}
  Basically, we will add $\outdegree{G}{s}$ chips to the sink $s$, so that it becomes firable. We therefore define $\overline{c^{*^{s}}} : V \to \mathbb N$ as the configuration such that $\overline{c^{*^{s}}}(v)=c^{*^s}(v)$ if $v \neq s$ and $\overline{c^{*^{s}}}(s)=\outdegree{G}{s}$. Note that we will always apply this operator to a configuration on which the sink is specified.
  
  Similarly, $\overline{c^{*^{s}}}^i : V \to \mathbb N$ denotes the configuration where we put $\outdegree{G}{s}$ plus $i$ extra chips on $s$, {\em i.e.} such that $\overline{c^{*^{s}}}^i(v)=c^{*^s}(v)$ if $v \neq s$ and $\overline{c^{*^{s}}}^i(s)=\outdegree{G}{s}+i$. Obviously, $\overline{c^{*^{s}}}=\overline{c^{*^{s}}}^0$.
  
  The sum of chips we are interested in may be applied to a configuration $c$ or $c^{*^s}$, and is defined as $\confdegree{G}{s}{c}=\sum \limits_{v \in V} \overline{c^{*^{s}}}(v)$.
\end{notation}

\begin{figure}
\centering
\subfloat[A configuration $c$ of $G$]{\label{fig:im7}\includegraphics[bb=0 0 304 136,width=2.05in,height=0.912in,keepaspectratio]{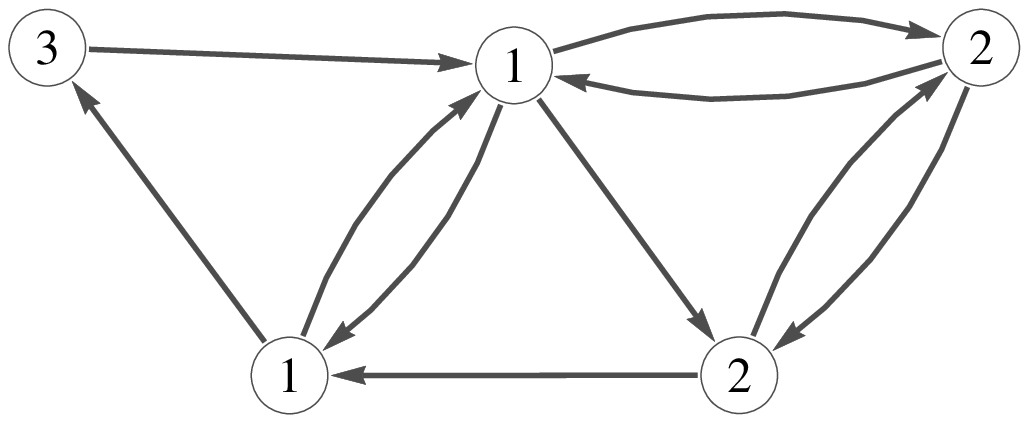}}\quad \quad
\subfloat[$\stab{G}{v_4}{c}$]{\label{fig:im8}\includegraphics[bb=0 0 296 133,width=2.05in,height=0.917in,keepaspectratio]{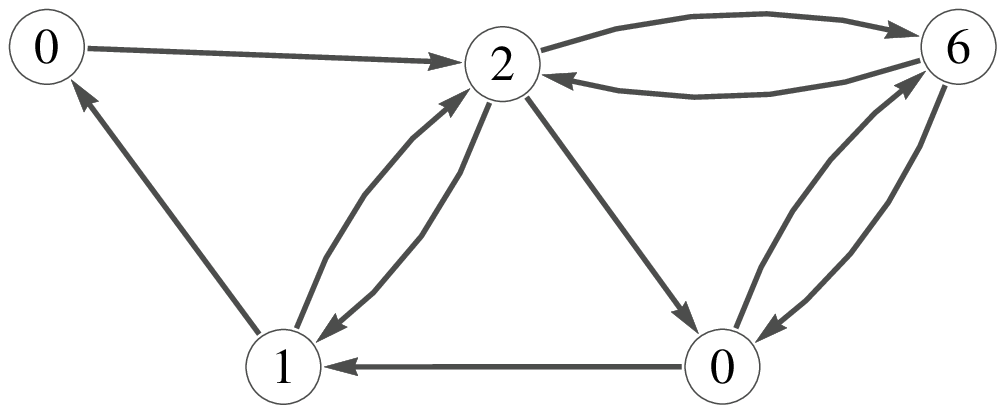}}
\caption{Chip-unvanished stabilization}
\label{fig:im0708}
\end{figure}

The map $\theta$ is based on the following property, which describes the construction of a configuration belonging to $\mathcal C_2$ from a configuration of $\mathcal C_1$. Note that the configurations of $\mathcal C_1$ (resp. $\mathcal C_2$) have type $V\backslash \set{s_1} \to \mathbb N$ (resp. $V \backslash \set{s_2} \to \mathbb N$) and are therefore denoted $c^{*^{s_1}}$ (resp. $c^{*^{s_2}}$). The procedure is straightforward: given a stable configuration of $\mathcal C_1$, we add $\outdegree{G}{s}$ chips to the sink and then stabilize according to the sink $s_2$. The resulting configuration, restricted to $V \backslash \set{s_2}$, belongs to $\mathcal C_2$.

\begin{lem}
\label{CFG:rectransformation}
Let $c^{*^{s_1}}\!\! \in \mathcal{C}_1$, then $\left(\overline{c^{*^{s_1}}}\right)^{\circ^{s_2}*^{s_2}}\!\! \in \mathcal{C}_2$. Moreover, $\left(\overline{c^{*^{s_1}}}\right)^{\circ^{s_2}}\!\!\!\!(s_2) \geq \outdegree{G}{s_2}$.
\end{lem}

Since the concept of this lemma is at the heart of the construction of the map $\theta$, we give an illustration of the claim before going into the details of the proof. We consider the Eulerian digraph given in Figure \ref{fig:im01020304} with $s_1=s$ and $s_2=v_3$. Figure \ref{fig:im9} shows a recurrent configuration with respect to sink $s$. The configuration $\overline{c^{*^s}}$ is given in Figure \ref{fig:im10}. Figure \ref{fig:im11} and Figure \ref{fig:im12} show $\left(\overline{c^{*^{s}}}\right)^{\circ^{v_3}}$ and its restriction to $V\backslash\set{v_3}$. Using the Burning algorithm, one easily checks that the configuration in Figure \ref{fig:im12} is indeed recurrent with respect to sink $v_3$.

\begin{figure}
\centering
\subfloat[A recurrent configuration $c^{*^s}$]{\label{fig:im9}\includegraphics[bb=0 0 296 132,width=1.88in,height=0.836in,keepaspectratio]{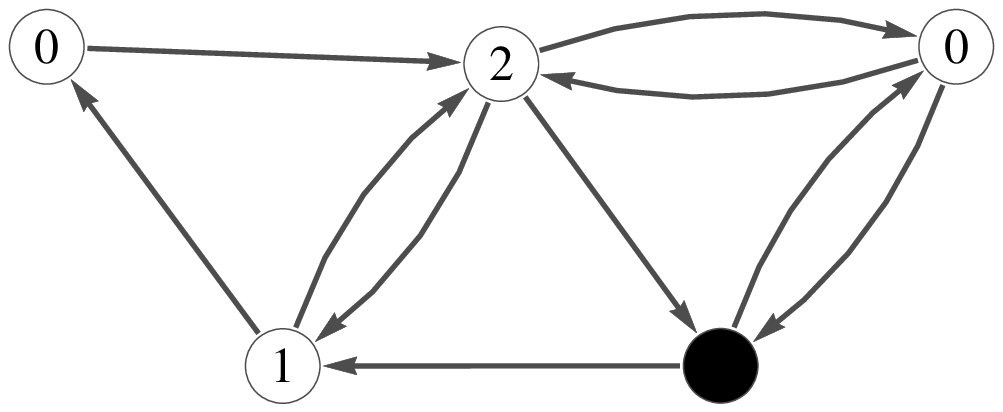}}\quad \quad
\subfloat[Configuration $\overline{c^{*^s}}$]{\label{fig:im10}\includegraphics[bb=0 0 277 124,width=1.88in,height=0.841in,keepaspectratio]{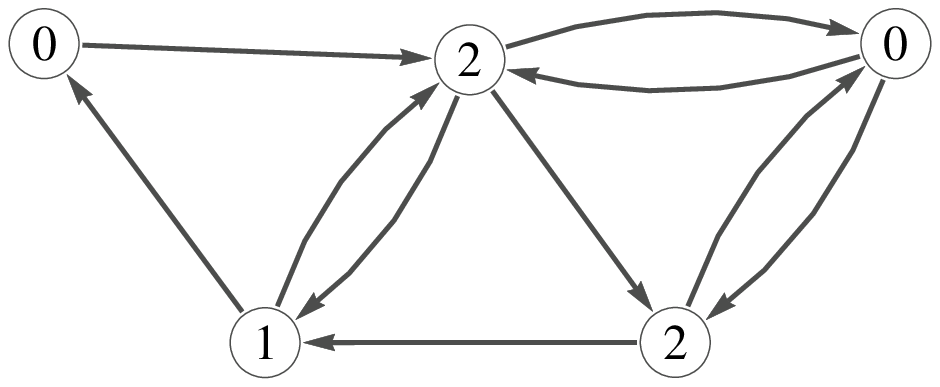}}\\
\subfloat[$\left(\overline{c^{*^{s}}}\right)^{\circ^{v_3}}$]{\label{fig:im11}\includegraphics[bb=0 0 286 128,width=1.88in,height=0.837in,keepaspectratio]{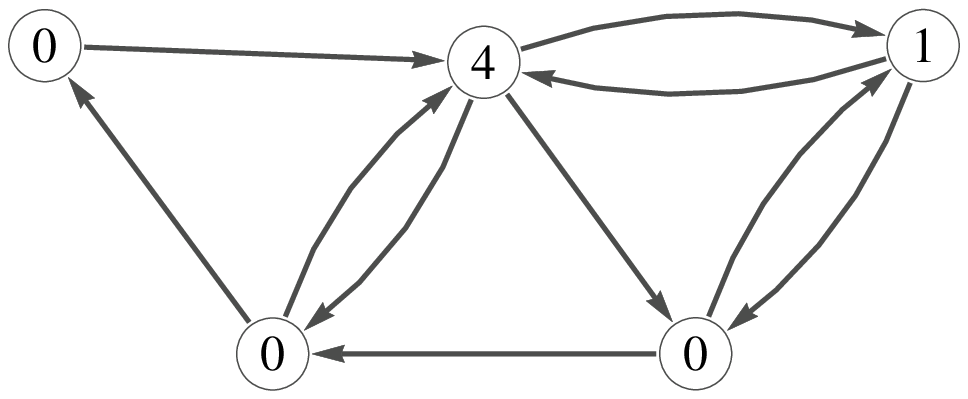}}\quad \quad
\subfloat[Recurrent configuration $\left(\overline{c^{*^{s}}}\right)^{\circ^{v_3}*^{v_3}}$]{\label{fig:im12}\includegraphics[bb=0 0 298 133,width=1.88in,height=0.837in,keepaspectratio]{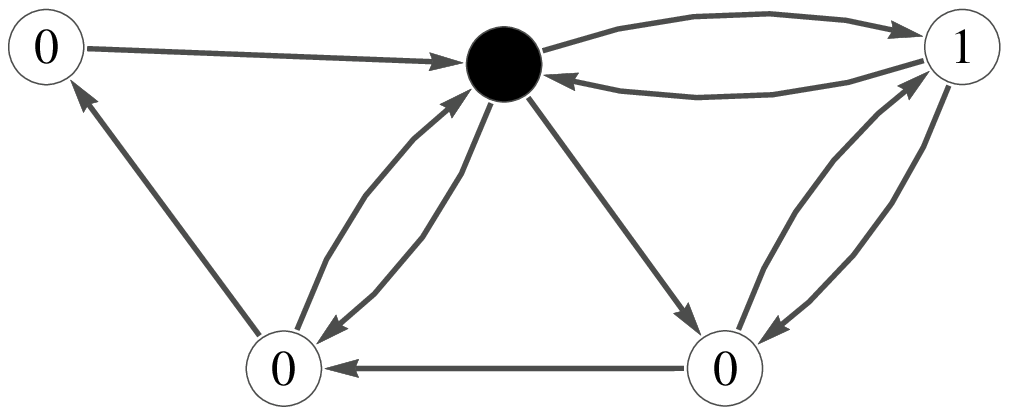}}
\caption{Change of sink}
\label{fig:im09101112}
\end{figure}

\begin{proof}[Proof of Lemma \ref{CFG:rectransformation}] We will once again use Lemma \ref{CFG:Dah} (the Burning algorithm), which provides a firing sequence associated to a recurrent configuration of $\mathcal C_1$, that we will manipulate to built a firing sequence associated to a recurrent configuration of $\mathcal C_2$. In $\overline{c^{*^{s_1}}}$, only $s_1$ is firable, and after firing it, we will use the firing sequence leading back to $c^{*^{s_1}}$, provided by Lemma \ref{CFG:Dah}.

Let $\beta_{s_1}^{*^{s_1}}:V\backslash\set{s_1}\to \mathbb{N}$ (resp. $\beta_{s_2}^{*^{s_2}}:V\backslash\set{s_2}\to \mathbb{N}$) be given by $\beta_{s_1}^{*^{s_1}}(v)$ (resp. $\beta_{s_2}^{*^{s_2}}(v)$) is equal to $\numarcs{G}{s_1}{v}$ (resp. $\numarcs{G}{s_2}{v}$). Let $c'$ be such that $\overline{c^{*^{s_1}}}\overset{s_1}{\to} c'$. We have ${c'}^{*^{s_1}}=c^{*^{s_1}}+\beta_{s_1}^{*^{s_1}}$, and can therefore apply\linebreak Lemma \ref{CFG:Dah} (since $c^{*^{s_1}}$ is recurrent with respect to sink $s_1$), providing a firing sequence $\mathfrak{f}=(v_1,v_2,\cdots,v_{|V|})$ of $\overline{c^{*^{s_1}}}$ such that $v_1=s_1$ and each vertex of $G$ occurs exactly once in this sequence. Let $k$ be such that $s_2=v_k$ and $d$ be the configuration reached after vertices $(v_1,v_2,\dots,v_{k-1})$ have been fired. Vertex $s_2$ is firable in $d$, thus $d(s_2) \geq \outdegree{G}{s_2}$. Since the game is convergent, and this intermediate configuration $d$ is reachable from $\overline{c^{*^{s_1}}}$ without firing $s_2$, when we stabilize $\overline{c^{*^{s_1}}}$ with respect to sink $s_2$ we end up with at least as many chips in $s_2$ as in $d$, and the second part of the lemma follows.

We now prove that $\left(\overline{c^{*^{s_1}}}\right)^{\circ^{s_2}*^{s_2}}\!\!=d^{\circ^{s_2} *^{s_2}}$ is a recurrent configuration of $\mathcal C_2$, by constructing a firing sequence in order to apply Lemma \ref{CFG:Dah}. Since $\overline{c^{*^{s_1}}} \overset{f}{\to} \overline{c^{*^{s_1}}}$, the sequence $\mathfrak{f}':=(s_2,v_{k+1},v_{k+2},\cdots,v_{|V|},v_1,v_2,\cdots,v_{k-1})$ is a firing sequence of $d$. We now consider the Chip-firing game with respect to sink $s_2$ (that is, on $\deleteoutarcs{G}{s_2}$), and the configuration $d^{*^{s_2}}$. Let $d'$ be such that $d\overset{s_2}{\to} d'$. We have $d'^{*^{s_2}}=d^{*^{s_2}}+\beta_{s_2}^{*^{s_2}}$, and the rest of the firing sequence implies that $d^{*^{s_2}}\!\!+\beta_{s_2}^{*^{s_2}} \overset{*}{\to} d^{*^{s_2}}$, therefore $d^{*^{s_2}}\!\!+n\beta_{s_2}^{*^{s_2}} \overset{*}{\to} d^{*^{s_2}}$ for any $n \in \mathbb N$. The recurrence of $\left(\overline{c^{*^{s_1}}}\right)^{\circ^{s_2}*^{s_2}}$ with respect to $s_2$ follows from the argument presented in the proof of Lemma \ref{CFG:maximum of total number of chips}.
\end{proof}

Lemma \ref{CFG:rectransformation} naturally suggests a bijection from $\mathcal{C}_1$ to $\mathcal{C}_2$ that is defined by $c^{*^{s_1}} \mapsto \left(\overline{c^{*^{s_1}}}\right)^{\circ^{s_2}*^{s_2}}$. However, this does not give the intended bijection since it does not necessarily preserves the $\textbf{\texttt{sum}}$ of chips
, as shown on Figure \ref{fig:im09101112}. The generalization of $\overline{c^{*^{s_1}}}$, denoted $\overline{c^{*^{s_1}}}^i$ and corresponding to adding some extra chips to $s_1$, is more flexible and can be used to improve the above map so that it preserves the $\textbf{\texttt{sum}}$ of chips. 
That is what we are going to present now.

The next lemma follows from the second item of Lemma \ref{CFG:stabilization of large configuration} and similar arguments as used in the proof of Lemma \ref{CFG:rectransformation}.

\begin{lem}
\label{CFG:generaltransformation}
For all $c^{*^{s_1}}\!\! \in \mathcal{C}_1$ and any $i \in \mathbb N$, we have $\left(\overline{c^{*^{s_1}}}^i\right)^{\circ^{s_2}*^{s_2}}\!\! \in \mathcal{C}_2$. Moreover, $\left(\overline{c^{*^{s_1}}}^i\right)^{\circ^{s_2}}\!\!\!\!(s_2)\geq \outdegree{G}{s_2}$.
\end{lem}

Lemma \ref{CFG:generaltransformation} produces a map from $\mathcal{C}_1 \times \mathbb{N}$ to $\mathcal{C}_2 \times \mathbb{N}$ defined by
$$\left(c^{*^{s_1}}\!\!,\,i\right) \mapsto \left(\left(\overline{c^{*^{s_1}}}^i\right)^{\circ^{s_2}*^{s_2}}\!\!,\,\left(\overline{c^{*^{s_1}}}^i\right)^{\circ^{s_2}}\!\!\!\!(s_2)-\outdegree{G}{s_2}\right).$$

This map is injective from the following result:

\begin{lem}
\label{inversibility}
Let $c^{*^{s_1}} \in \mathcal{C}_1$ and $i \in \mathbb{N}$, then $\left( \left( \overline{c^{*^{s_1}}}^i \right)^{\circ^{s_2}} \right)^{\circ^{s_1}}\!\!=\overline{c^{*^{s_1}}}^i$.
\end{lem}

\begin{proof}
For convenience, let $d$ denote $\left( \overline{c^{*^{s_1}}}^i \right)^{\circ^{s_2}}$. It follows from Lemma \ref{CFG:generaltransformation} that $d$ restricted to $V \backslash \set{s_2}$ belongs to $\mathcal{C}_2$ and $d(s_2) \geq \outdegree{G}{s_2}$. As a consequence, some firings happen when stabilizing $d$ according to the sink $s_1$. Let us prove that this process leads to $d^{\circ^{s_1}}=\overline{c^{*^{s_1}}}^i$. From Lemma \ref{CFG:generaltransformation} (with the application to $\mathcal C_2$), we have $d^{\circ^{s_1}*^{s_1}}\!\! \in \mathcal C_1$. Since $G$ is Eulerian, the configurations $d^{*^{s_1}}$ and $c^{*^{s_1}}$ belong to the same equivalence class, so do $d^{\circ^{s_1}*^{s_1}}$ and $c^{*^{s_1}}$. Both are recurrent, hence from Lemma \ref{CFG:recurrent:property:algebraic} they are equal. Finally, $d^{\circ^{s_1}}$ and $\overline{c^{*^{s_1}}}^i$ obviously contain the same total number of chips, and are equal on the vertices different from $s_1$, consequently they also contain the same number of chips on $s_1$.
\end{proof}

The aim is now to find, for every recurrent configuration $c^{*^{s_1}}\!\! \in \mathcal C_1$, the good $i$ so as to get a bijection from $\mathcal C_1$ to $\mathcal C_2$ that preserves the $\textbf{\texttt{sum}}$ of chips. We first concentrate on the $\textbf{\texttt{sum}}$ conservation: if one wants to have
$$\confdegree{G}{s_1}{c^{*^{s_1}}} = \textbf{\texttt{sum}}\,_{G,s_2}\left(\left(\overline{c^{*^{s_1}}}^i\right)^{\circ^{s_2}*^{s_2}}\right),$$
then the number $i$ must be chosen so that $\left(\overline{c^{*^{s_1}}}^i\right)^{\circ^{s_2}}\!\!\!\!(s_2)=\outdegree{G}{s_2}+i$, because the $i$ extra chips are not counted in both $\textbf{\texttt{sum}}$ in this case. The following shows that such an $i$ always exists.

\begin{lem}
\label{bijection:iexistence}
For every $c^{*^{s_1}}\!\! \in \mathcal{C}_1$ there exists $i\in \mathbb{N}$ such that $\left(\overline{c^{*^{s_1}}}^i\right)^{\circ^{s_2}}\!\!\!\!(s_2)=\outdegree{G}{s_2}+i$.
\end{lem} 


\begin{proof}
Let the function $f:\mathbb{N}\to \mathbb{Z}$ be defined by $f(i)=\stab{G}{s_2}{\left(\sinkaug{c}{s_1}{i}\right)}\!\!\!\!(s_2)-\outdegree{G}{s_2}-i$. We are going to prove that there exists $j \in \mathbb{N}$ such that $f(j)=0$. Since $\sinkaug{c}{s_1}{i}(v)\leq \sinkaug{c}{s_1}{i+1}(v)$ for every $v\neq s_2$, it follows from Lemma \ref{CFG:stabilization of large configuration} (using the trick $\sum \limits_{v \in V\backslash \set{s}}\!\! c(v) = \sum \limits_{v \in V} c(v) - c(s)$) that $\stab{G}{s_2}{\left(\sinkaug{c}{s_1}{i}\right)}\!\!\!\!(s_2)-c^{*^{s_1}}(s_2) \,\leq\, \stab{G}{s_2}{\left(\sinkaug{c}{s_1}{i+1}\right)}\!\!\!\!(s_2)-c^{*^{s_1}}(s_2)$, therefore $\stab{G}{s_2}{\left(\sinkaug{c}{s_1}{i}\right)}\!\!\!\!(s_2)\leq \stab{G}{s_2}{\left(\sinkaug{c}{s_1}{i+1}\right)}\!\!\!\!(s_2)$. As a consequence $f(i+1)-f(i)\geq -1$ for every $i$, that is, the function $f$ decreases by at most one.

By Lemma \ref{CFG:generaltransformation} we have $\stab{G}{s_2}{\left(\overline{c^{*^{s_1}}}\right)}\!\!(s_2)\geq \outdegree{G}{s_2}$, therefore $f(0)\geq 0$. Since $f(i+1)-f(i)\geq -1$ for any $i\in \mathbb{N}$, the proof is completed by showing that there is $j\in \mathbb{N}$ such that $f(j)\leq 0$. In particular, we are going to prove that $f(N-1)\leq 0$, where $N=|\mathcal{C}_2|$. Note that $N$ is the order of the Sandpile group of $G$ with respect to sink $s_2$.

$f(N-1)= \stab{G}{s_2}{\left(\sinkaug{c}{s_1}{N-1}\right)}\!\!(s_2)-\outdegree{G}{s_2}-(N-1)$. We are going to use Lemma \ref{CFG:maximum of total number of chips}, which states that the recurrent configuration has maximum total number of chips over stable configurations of its equivalence class, in order to upper bound $\stab{G}{s_2}{\left(\sinkaug{c}{s_1}{N-1}\right)}\!\!(s_2)$ by $c^{*^{s_1}}(s_2)+N$, and the result follows since vertex $s_2$ is stable in the recurrent configuration $c^{*^{s_1}}$ (meaning that $c^{*^{s_1}}(s_2) \leq \outdegree{G}{s_2} - 1$).

Let $\textbf{1}_{s_1}:V\to \mathbb{N}$ be given by $\textbf{1}_{s_1}(v)=0$ if $v\neq s_1$ and $\textbf{1}_{s_1}(s_1)=1$. We have $\sinkaug{c}{s_1}{N-1}=\sinkaug{c}{s_1}{(-1)}\!\!+N\,\textbf{1}_{s_1}$ (note that $\sinkaug{c}{s_1}{(-1)}\!\!:V \to \mathbb N$ is a configuration since our digraph is Eulerian and has a global sink), thus the choice of $N$ implies that $\sinkaug{c}{s_1}{N-1}$ and $\sinkaug{c}{s_1}{(-1)}$ are in the same equivalence class with respect to sink $s_2$, and the first contains $N$ more chips than the latter. The configuration $\stab{G}{s_2}{\left(\sinkaug{c}{s_1}{N-1}\right)}$ is recurrent, hence from Lemma \ref{CFG:maximum of total number of chips} we have $\sum \limits_{v \neq s_2} \sinkaug{c}{s_1}{(-1)}\!(v) \,\leq \sum \limits_{v \neq s_2} \stab{G}{s_2}{\left(\sinkaug{c}{s_1}{N-1}\right)}\!\!(v)$. It remains to exploit the total number of chips difference between the two configurations: $\sum \limits_{v \in V} \sinkaug{c}{s_1}{(-1)}\!(v) +N= \sum \limits_{v \in V} \stab{G}{s_2}{\left(\sinkaug{c}{s_1}{N-1}\right)}\!\!(v)$. Replacing $\sum \limits_{v \neq s_2} \!\!x(v)$ by $\sum \limits_{v \in V} \!x(v) - x(s_2)$ on both sides, the inequality given by Lemma \ref{CFG:maximum of total number of chips} thus becomes $\sinkaug{c}{s_1}{(-1)}\!(s_2) +N \geq \stab{G}{s_2}{\left(\sinkaug{c}{s_1}{N-1}\right)}\!\!(s_2)$, and equivalently $c^{*^{s_1}}(s_2) +N \geq \stab{G}{s_2}{\left(\sinkaug{c}{s_1}{N-1}\right)}\!\!(s_2)$.
\end{proof}

We can now construct the intended bijection $\theta$. For each $c^{*^{s_1}} \!\!\in \mathcal C_1$, let $\augnum{s_2}{c^{*^{s_1}}}$ denote the smallest number $i\in \mathbb{N}$ such that $\left(\overline{c^{*^{s_1}}}^i\right)^{\circ^{s_2}}\!\!\!\!(s_2)=\outdegree{G}{s_2}+i$. The positive integer $\augnum{s_2}{c^{*^{s_1}}}$ is called the {\em swap number} of $c^{*^{s_1}}$ from $s_1$ to $s_2$. By Lemma \ref{bijection:iexistence} we know that swap numbers are well-defined and unique.
$$\begin{array}{rrcl}
  \theta : & \mathcal C_1 & \to & \mathcal C_2\\
  & c^{*^{s_1}} & \mapsto & \left(\overline{c^{*^{s_1}}}^{\augnum{s_2}{c^{*^{s_1}}}}\right)^{\circ^{s_2}*^{s_2}}
\end{array}$$
The map $\theta$ verifies
$$\confdegree{G}{s_1}{c^{*^{s_1}}}=\confdegree{G}{s_2}{\theta(c^{*^{s_1}})}.$$
Since $|\mathcal C_1|=|\mathcal C_2|$ is finite, in order to prove Theorem \ref{CFG:sink-independence} it remains to show that $\theta$ is injective. Let us first present some properties of swap numbers. A configuration $c^{*^{s_1}} \!\!\in \mathcal{C}_1$ is called \emph{minimal} if there is no configuration $c'^{*^{s_1}} \!\!\in \mathcal{C}_1$ such that $c^{*^{s_1}} \!\neq c'^{*^{s_1}}$ and $c'^{*^{s_1}}(v) \leq c^{*^{s_1}}(v)$ for all $v\neq s_1$, and \emph{minimum} if $\underset{v\neq s_1}{\sum}c^{*^{s_1}}(v)$ is minimum over all configurations in $\mathcal{C}_1$.

\begin{prop}
\label{augment number of minimum recurrent configuration}
Let $c^{*^{s_1}}\!\!\in \mathcal{C}_1$. If $c^{*^{s_1}}$ is minimum then $\augnum{s_2}{c^{*^{s_1}}}=0$. 
\end{prop}

\begin{proof}
By definition of $\augnum{s_2}{c^{*^{s_1}}}$, the aim is to prove that $\left(\overline{c^{*^{s_1}}}\right)^{\circ^{s_2}}\!\!\!\!(s_2)=\outdegree{G}{s_2}$. The proof relies intuitively on Lemma \ref{CFG:rectransformation}: it says that the lower bound for $\left(\overline{c^{*^{s_1}}}\right)^{\circ^{s_2}}\!\!\!\!(s_2)$ is always reached for the configuration containing the minimum total number of chips. Let us assume it is false, therefore we can define $d \neq \left(\overline{c^{*^{s_1}}}\right)^{\circ^{s_2}}$ such that $d(v)=\left(\overline{c^{*^{s_1}}}\right)^{\circ^{s_2}}\!\!\!\!(v)$ for $v \neq s_2$ and $d(s_2)= \outdegree{G}{s_2}$, and we will get a contraction to the minimality of $c^{*^{s_1}}$.

It follows from Lemma \ref{CFG:rectransformation} and the assumption that $\left(\overline{c^{*^{s_1}}}\right)^{\circ^{s_2}}\!\!\!\!(s_2) > \outdegree{G}{s_2}$, therefore $\underset{v\in V}{\sum}d(v)<\underset{v\in V}{\sum} \overline{c^{*^{s_1}}}(v)$. By Lemma \ref{CFG:rectransformation} the configuration $\left(\overline{c^{*^{s_1}}}\right)^{\circ^{s_2}*^{s_2}} \in \mathcal{C}_2$, and so does $d^{*^{s_2}}$. Applying again Lemma \ref{CFG:rectransformation} to $\mathcal{C}_2$ and $s_2$, we have $d^{\circ^{s_1}*^{s_1}} \in \mathcal{C}_1$ and $d^{\circ^{s_1}}(s_1)\geq \outdegree{G}{s_1}=\overline{c^{*^{s_1}}}(s_1)$. Now, since $\underset{v\in V}{\sum}d^{\circ^{s_1}}(v)=\underset{v\in V}{\sum} d(v) < \underset{v\in V}{\sum}\overline{c^{*^{s_1}}}(v)$, it follows that $\underset{v\neq s_1}{\sum} d^{\circ^{s_1}}(v)<\underset{v\neq s_1}{\sum} c^{*^{s_1}}(v)$, a contradiction to the minimality of $c^{*^{s_1}}$.
\end{proof}

\begin{prop}
\label{compatibility with the configuration order}
Let $c^{*^{s_1}}\!\!,c'^{*^{s_1}}\!\!\in \mathcal{C}_1$. If $c^{*^{s_1}}(v)\leq c'^{*^{s_1}}(v)$ for any $v \neq s_1$, then $\augnum{s_2}{c^{*^{s_1}}}\leq \augnum{s_2}{c'^{*^{s_1}}}$.
\end{prop}

\begin{proof}
Regarding Proposition \ref{augment number of minimum recurrent configuration}, we would intuitively expect that the number $\augnum{s_2}{x}$ increases monotonously with the total number of chips of $x$. We use the same construction as in the proof of Lemma \ref{bijection:iexistence}.

Let $k$ denote $\augnum{s_2}{c'^{*^{s_1}}}$. Let $d$ and $d'$ denote $\stab{G}{s_2}{\sinkaug{c}{s_1}{k}}$ and $\stab{G}{s_2}{\sinkaug{c'}{s_1}{k}}$, respectively. Since $\sinkaug{c}{s_1}{k}(v)\leq \sinkaug{c'}{s_1}{k}(v)$ for any $v\neq s_2$, it follows from Lemma \ref{CFG:stabilization of large configuration} that $d(s_2)-c^{*^{s_1}}(s_2)\leq d'(s_2)-c'^{*^{s_1}}(s_2)$. Since by hypothesis $c^{*^{s_1}}(s_2)\leq c'^{*^{s_1}}(s_2)$, we have $d(s_2)\leq d'(s_2)$, therefore $d(s_2)-\outdegree{G}{s_2}-k\leq d'(s_2)-\outdegree{G}{s_2}-k=0$. Let $f:\mathbb{N}\to \mathbb{Z}$ be given by $f(i)=\stab{G}{s_2}{\sinkaug{c}{s_1}{i}}\!\!\!\!(s_2)-i-\outdegree{G}{s_2}$, from the previous inequality it verifies that $f(k)\leq 0$, and $\augnum{s_2}{c^{*^{s_1}}}$ is by definition the smallest $j$ such that $f(j)=0$. By the same arguments as in the proof of Lemma \ref{bijection:iexistence}, we have $f(0)\geq 0$ and $f(i+1)-f(i)\geq -1$ for any $i\in \mathbb{N}$. As a consequence, there is $j\in [0..k]$ such that $f(j)=0$, therefore $\augnum{s_2}{c^{*^{s_1}}} \leq k=\augnum{s_2}{c'^{*^{s_1}}}$.
\end{proof}

By Propositions \ref{augment number of minimum recurrent configuration} and \ref{compatibility with the configuration order}, for a recurrent configuration $c^{*^{s_1}}\!\!\in \mathcal C_1$ the number $\augnum{s_2}{c^{*^{s_1}}}$ increases\linebreak monotonously as we add chips to $c^{*^{s_1}}$, starting from $0$ when the configuration is minimum. One tends to think that a minimal configuration $c^{*^{s_1}}$ should also have $\augnum{s_2}{c^{*^{s_1}}}=0$, but it may indeed be strictly positive as shown on Figure \ref{fig:im1314}.

\begin{figure}
\centering
\subfloat[An Eulerian graph]{\includegraphics[bb=0 0 280 286,width=1.62in,height=1.66in,keepaspectratio]{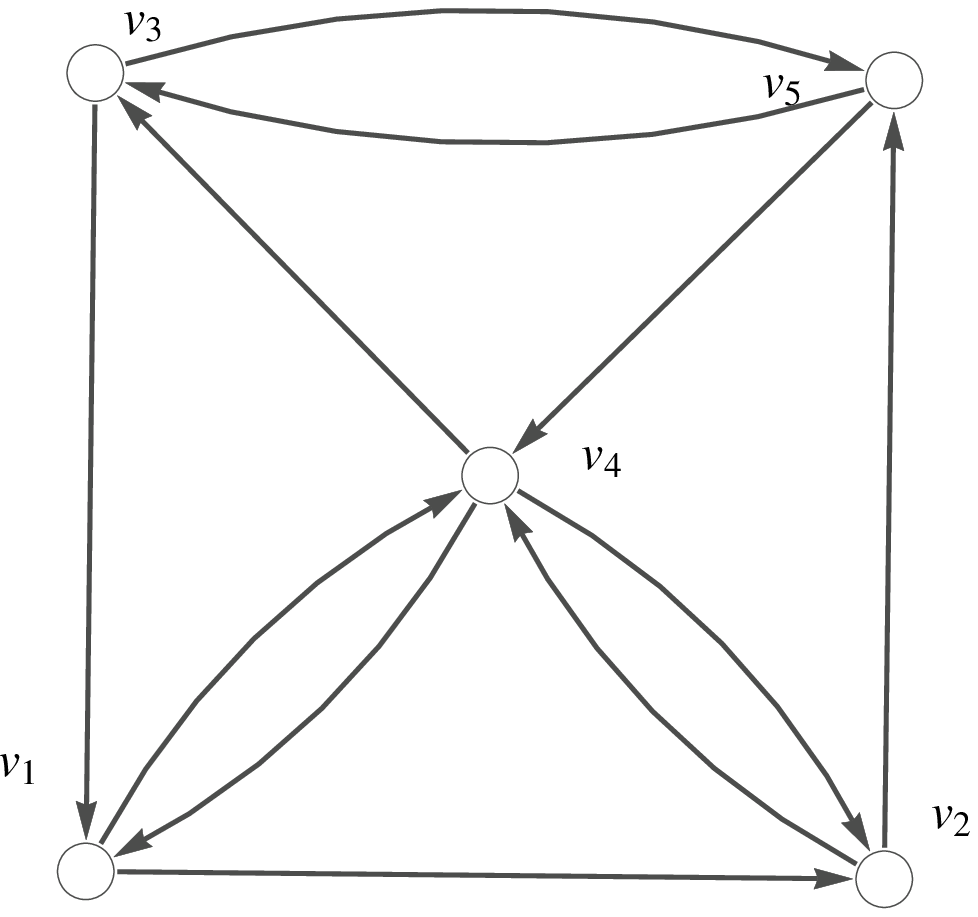}}\quad \quad
\subfloat[A minimal recurrent configuration with respect to sink $v_5$]{\includegraphics[bb=0 0 185 196,width=1.45in,height=1.54in,keepaspectratio]{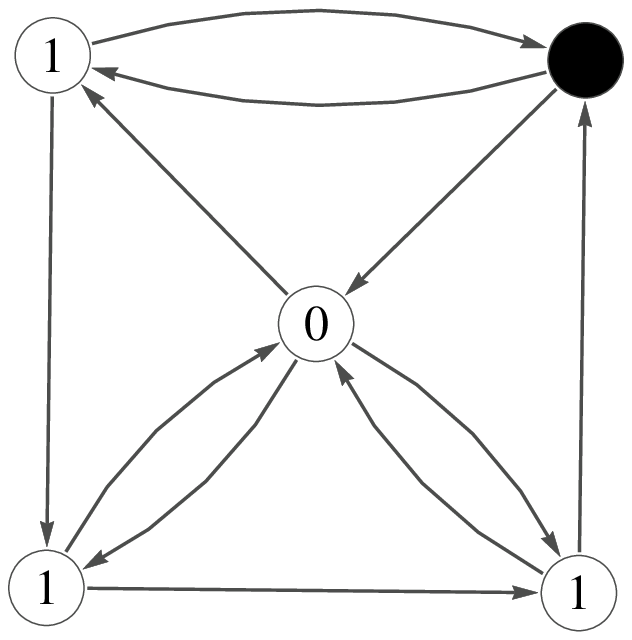}}
\caption{A minimal recurrent configuration $c^{*^{v_5}}$ with $\augnum{v_4}{c^{*^{v_5}}}=1$}
\label{fig:im1314}
\end{figure}

\begin{ques}
Give an upper bound of $\augnum{s_2}{c^{*^{s_1}}}$ when $c^{*^{s_1}}$ is minimal.
\end{ques}

There is a nice relation between swap numbers for $c^{*^{s_1}}$ from $s_1$ to $s_2$, and for $\theta(c^{*^{s_1}})$ from $s_2$ to $s_1$. The following proposition does most part of the work to prove Theorem \ref{CFG:sink-independence}, and the latter can be considered as a corollary of this result.

\begin{prop}
\label{swapping swap number}
For all $c^{*^{s_1}} \!\!\in \mathcal C_1$, we have $\augnum{s_1}{\theta(c^{*^{s_1}})}=\augnum{s_2}{c^{*^{s_1}}}$.
\end{prop}

\begin{proof}
The proposition is proved with two inequalities.
\begin{itemize}
\item $\augnum{s_1}{\theta(c^{*^{s_1}})} \leq \augnum{s_2}{c^{*^{s_1}}}$:
  
We consider $c'=\stab{G}{s_2}{\left( \sinkaug{c}{s_1}{\augnum{s_2}{c^{*^{s_1}}}} \right)}$, that is, $c'^{*^{s_2}}\!\!=\theta(c^{*^{s_1}})$. First, by definition of $\augnum{s_2}{c^{*^{s_1}}}$ we have $c'(s_2)=\outdegree{G}{s_2}+\augnum{s_2}{c^{*^{s_1}}}$, therefore $c'=\overline{\theta(c^{*^{s_1}})}^{\augnum{s_2}{c^{*^{s_1}}}}$. Second, by Lemma \ref{inversibility} applied to $c^{*^{s_1}}$, we have
$$c'^{\circ^{s_1}}= \left(\left( \sinkaug{c}{s_1}{\augnum{s_2}{c^{*^{s_1}}}} \right)^{\circ^{s_2}}\right)^{\circ^{s_1}}=\sinkaug{c}{s_1}{\augnum{s_2}{c^{*^{s_1}}}} \text{, thus } c'^{\circ^{s_1}*^{s_1}}=c^{*^{s_1}}.$$
As a consequence,
$$\confdegree{G}{s_2}{\theta(c^{*^{s_1}})}=\confdegree{G}{s_1}{c^{*^{s_1}}}=\confdegree{G}{s_1}{c'^{\circ^{s_1}*^{s_1}}}=\textbf{\texttt{sum}}\,_{G,s_1}\left(\left( \overline{\theta(c^{*^{s_1}})}^{\augnum{s_2}{c^{*^{s_1}}}} \right)^{\circ^{s_1}*^{s_1}}\right).$$
It follows that $\augnum{s_1}{\theta(c^{*^{s_1}})} \leq \augnum{s_2}{c^{*^{s_1}}}$, because $\augnum{s_1}{\theta(c^{*^{s_1}})}$ is the minimal number $i$ such that $\left( \overline{\theta(c^{*^{s_1}})}^i \right)^{\circ^{s_1}}\!\!\!\!(s_1)=\outdegree{G}{s_1}+i$ and $\augnum{s_2}{c^{*^{s_1}}}$ is such a number.

\item $\augnum{s_1}{\theta(c^{*^{s_1}})} \geq \augnum{s_2}{c^{*^{s_1}}}$:

This part of the lemma is more involved. For convenience, let us denote $\theta=\theta(c^{*^{s_1}})$. The previous inequality implies $\augnum{s_1}{\theta} \leq \augnum{s_2}{c^{*^{s_1}}}$. In order to get a contradiction, let us suppose that $\augnum{s_1}{\theta} < \augnum{s_2}{c^{*^{s_1}}}$. Let ${c'}^{*^{s_1}} \in \mathcal{C}_1$ be such that $\left({\bar{\theta}}^{\mathcal{I}_{s_1}(\theta)}\right)^{\circ^{s_1}}=\overline{c'^{*^{s_1}}}^{\mathcal{I}_{s_1}(\theta)}$ (the existence of $c'^{*^{s_1}}$ is due to Lemma \ref{CFG:rectransformation} and the definition of $\mathcal{I}_{s_1}(\theta)$). The above inequality applied to $\theta$ implies that $c'^{*^{s_1}}\neq c^{*^{s_1}}$.

We have
\begin{align}
\label{eq:cc'}
\left( \sinkaug{c}{s_1}{\augnum{s_2}{c^{*^{s_1}}}} \right)^{\circ^{s_2}} \!\!=~ \overline{\theta}^{\augnum{s_2}{c^{*^{s_1}}}} \text{~ and ~~} \left( \sinkaug{c'}{s_1}{\augnum{s_1}{\theta}} \right)^{\circ^{s_2}} \!\!=~ \overline{\theta}^{\augnum{s_1}{\theta}},
\end{align}
therefore the two configurations
$$\left( \sinkaug{c}{s_1}{\augnum{s_2}{c^{*^{s_1}}}} \right)^{*^{s_2}} \text{and ~~} \left( \sinkaug{c'}{s_1}{\augnum{s_1}{\theta}} \right)^{*^{s_2}}$$
are in the same equivalence class for the CFG with sink $s_2$. Removing $\augnum{s_1}{\theta}$ chips to $s_1$ in both configurations does not affect the equivalence relation, hence with $k = \augnum{s_2}{c^{*^{s_1}}} - \augnum{s_1}{\theta} > 0$,
$$\left( \sinkaug{c}{s_1}{k} \right)^{*^{s_2}} \text{and ~~} \left( \sinkaug{c'}{s_1}{} \right)^{*^{s_2}}$$
are also in the same equivalence class for the CFG with sink $s_2$, and from Lemma \ref{CFG:generaltransformation} and the unicity of the recurrent configuration in an equivalence class (Lemma \ref{CFG:recurrent:property:algebraic}),
\begin{align}
\label{eq:cc's2}
\left( \sinkaug{c}{s_1}{k} \right)^{\circ^{s_2}*^{s_2}} =~~ \left( \sinkaug{c'}{s_1}{} \right)^{\circ^{s_2}*^{s_2}}.
\end{align}
From equation $(\ref{eq:cc'})$ there are $k$ more chips in $\sinkaug{c}{s_1}{\augnum{s_2}{c^{*^{s_1}}}}$ than in $\sinkaug{c'}{s_1}{\augnum{s_1}{\theta}}$, thus it follows from the above equality $(\ref{eq:cc's2})$ that
\begin{align}
\label{eq:contradiction}
\left( \sinkaug{c}{s_1}{k} \right)^{\circ^{s_2}}\!\!\!\!(s_2) ~=~ \left( \sinkaug{c'}{s_1}{} \right)^{\circ^{s_2}}\!\!\!\!(s_2) + k
\end{align}

We now consider the two configurations $e$ and $e'$ defined by
$$\begin{array}{rc}
& e = \sinkaug{c}{s_1}{k} - \textbf{1}_{s_1}\\[.2em]
\iff & e + \textbf{1}_{s_1} = \sinkaug{c}{s_1}{k}
\end{array}
\text{~~~ and ~~}
\begin{array}{rc}
& e' = \sinkaug{c'}{s_1}{} - \textbf{1}_{s_1}\\[.2em]
\iff & e' + \textbf{1}_{s_1} = \sinkaug{c'}{s_1}{}
\end{array}$$
with $\textbf{1}_{s_1}$ the configuration having $1$ chip in $s_1$ and none in other vertices. It follows from equality $(\ref{eq:cc's2})$ that
\begin{align}
\label{eq:ee's2}
\left(e^{\circ^{s_2}}+\textbf{1}_{s_1}\right)^{\circ^{s_2}*^{s_2}} =~ \left(e'^{\circ^{s_2}}+\textbf{1}_{s_1}\right)^{\circ^{s_2}*^{s_2}}.
\end{align}
As we will see, it is not possible that both:
\begin{itemize}
  \item those two configurations are equal;
  \item enough chips go to $s_2$ during those stabilization processes so that equation $(\ref{eq:contradiction})$ is verified.
\end{itemize}
Let us present a reasoning contradicting equation $(\ref{eq:contradiction})$.

We first work on the total chip content of $e^{\circ^{s_2}*^{s_2}}$ and $e'^{\circ^{s_2}*^{s_2}}$. For the same reason as above, $e^{*^{s_2}}$ and $e'^{*^{s_2}}$ belong to the same equivalence class for the CFG with sink $s_2$, and so do $e^{\circ^{s_2}*^{s_2}}$ and $e'^{\circ^{s_2}*^{s_2}}$ because the firing process does not affect the equivalence relation. $e^{\circ^{s_2}*^{s_2}}=\left( \sinkaug{c}{s_1}{k-1} \right)^{\circ^{s_2}*^{s_2}}$ with $k-1 \geq 0$, thus form Lemma \ref{CFG:generaltransformation} it is recurrent. Furthermore $e'^{\circ^{s_2}*^{s_2}}$ is stable, and since they belong to the same equivalence class, it follows from Lemma \ref{CFG:maximum of total number of chips} that
\begin{align}
\label{eq:totalchips}
\sum \limits_{v \neq s_2} e^{\circ^{s_2}*^{s_2}} (v) ~~~\geq~~ \sum \limits_{v \neq s_2} e'^{\circ^{s_2}*^{s_2}} (v).
\end{align}
Now we compare the number of chips going into the sink $s_2$. Let $\mathfrak f=(v_1,\dots,v_p)$ and $\mathfrak f'=(w_1,\dots,w_{p'})$ be two firing sequences such that
$$\left(e^{\circ^{s_2}}+\textbf{1}_{s_1}\right) \overset{\mathfrak f}{\longrightarrow} \left(e^{\circ^{s_2}}+\textbf{1}_{s_1}\right)^{\circ^{s_2}} \text{~ and ~~} \left(e'^{\circ^{s_2}}+\textbf{1}_{s_1}\right) \overset{\mathfrak f'}{\longrightarrow} \left(e'^{\circ^{s_2}}+\textbf{1}_{s_1}\right)^{\circ^{s_2}}.$$
Obviously $s_2 \notin \mathfrak f$ and $s_2 \notin \mathfrak f'$, and it follows from equations $(\ref{eq:ee's2})$ and $(\ref{eq:totalchips})$ that during the stabilization process, more chips goes to $s_2$ in $\mathfrak f$ than in $\mathfrak f'$:
\begin{align}
\label{eq:totaldeg}
\sum \limits_{1 \leq i \leq p} \numarcs{G}{v_i}{s_2} ~~~~\geq~~ \sum \limits_{1 \leq i \leq p'} \numarcs{G}{w_i}{s_2}
\end{align}
In order to get the intended contradiction with $(\ref{eq:contradiction})$, let us have a close look at the chip content in both sinks $s_2$, using the fact that from the minimality of $\augnum{s_2}{c^{*^{s_1}}}$,
$$\left( \sinkaug{c}{s_1}{k-1} \right)^{\circ^{s_2}}\!\!(s_2) ~~>~~ \outdegree{G}{s_2} + (k-1).$$
$$\begin{array}{rcl}
\left( \sinkaug{c}{s_1}{k} \right)^{\circ^{s_2}}\!\!\!\!(s_2) & = & \left(e^{\circ^{s_2}}+\textbf{1}_{s_1}\right)^{\circ^{s_2}}\!\!(s_2)\\[.2em]
& = & e^{\circ^{s_2}}(s_2) + \sum \limits_{1 \leq i \leq p} \numarcs{G}{v_i}{s_2}\\[1em]
& = & \left( \sinkaug{c}{s_1}{k-1} \right)^{\circ^{s_2}}\!\!\!\!(s_2) + \sum \limits_{1 \leq i \leq p} \numarcs{G}{v_i}{s_2}\\[1em]
& > & \outdegree{G}{s_2} + (k-1) + \sum \limits_{1 \leq i \leq p} \numarcs{G}{v_i}{s_2}\\[1em]
& \underset{\text{equation }(\ref{eq:totaldeg})}{\geq} & \outdegree{G}{s_2} + (k-1) + \sum \limits_{1 \leq i \leq p'} \numarcs{G}{w_i}{s_2}\\[1em]
& \underset{\text{stability}}{\geq} & \outdegree{G}{s_2} + (k-1) + \sum \limits_{1 \leq i \leq p'} \numarcs{G}{w_i}{s_2} + e'^{\circ^{s_2}}(s_2) - \outdegree{G}{s_2} + 1\\[1em]
& = & k + \left( e'^{\circ^{s_2}} + \textbf{1}_{s_1} \right)^{\circ^{s_2}}\!\!(s_2)\\[.5em]
& = & k + \left( \sinkaug{c'}{s_1}{} \right)^{\circ^{s_2}}\!\!\!\!(s_2)
\end{array}$$
which contradicts equation $(\ref{eq:contradiction})$.\\[-2em]

\end{itemize}
\end{proof}

Theorem \ref{CFG:sink-independence} is now easy to prove.

\begin{proof}[Proof of Theorem \ref{CFG:sink-independence}.]
Since $|\mathcal{C}_1|=|\mathcal{C}_2|$, it remains to prove that the map $\theta$ is injective. For a contradiction, suppose it is not, that is, there exist $c^{*^{s_1}}$ and $c'^{*^{s_1}}$ belonging to $\mathcal{C}_1$ and such that
$$c^{*^{s_1}} \!\!\neq c'^{*^{s_1}} \text{ ~~and~~~ } \theta(c^{*^{s_1}})=\theta(c'^{*^{s_1}}).$$
By Proposition \ref{swapping swap number} we have $\augnum{s_2}{c^{*^{s_1}}}=\augnum{s_2}{c'^{*^{s_1}}}$, and Lemma \ref{inversibility} implies that
$$c^{*^{s_1}}=~\left( \overline{\theta(c^{*^{s_1}})}^{\augnum{s_2}{c^{*^{s_1}}}} \right)^{\circ^{s_1}*^{s_1}}=~~\left( \overline{\theta(c'^{*^{s_1}})}^{\augnum{s_2}{c'^{*^{s_1}}}} \right)^{\circ^{s_1}*^{s_1}}=~c'^{*^{s_1}},$$
a contradiction.
\end{proof}

\section{Tutte-like properties of generating function of recurrent configurations}
\label{tutte-like section}

We present in this section a natural generalization of the partial Tutte polynomial in one variable, for the class of Eulerian digraphs. In order to set up the most general setting, we introduce it to the class of Eulerian digraphs with loops. Note that loops are not interesting regarding the Chip-firing game: a loop simply ``freezes'' a chip on one vertex, that is the reason why we did not consider them in previous sections. The Tutte-like polynomial we present is constructed from the generating function of the set of recurrent configurations with respect to an arbitrary sink, and its unicity is based on the sink-independence property exposed in Theorem \ref{CFG:sink-independence}.

Let us first present the extension of Theorem \ref{CFG:sink-independence} to the class of Eulerian digraphs with loops. We begin with the definition of the Chip-firing game for this class of graphs. Note that the out-degree of a vertex $v$ is the number of arcs whose tail is $v$, therefore includes loops. Let $G=(V,A)$ be an Eulerian digraph possibly having loops. A vertex $v$ is \emph{firable} in a configuration $c$ if $c(v)\geq \outdegree{G}{v}$ and $\outdegree{G}{v}-\numarcs{G}{v}{v}\geq 1$, where $\numarcs{G}{v}{v}$ is the number of loops at $v$. Firing a firable vertex $v$ means the process that decreases $c(v)$ by $\outdegree{G}{v}$ and increases each $c(v')$ by $\numarcs{G}{v}{v'}$ for all $v'$, or equivalently decreases $c(v)$ by $\outdegree{G}{v}-\numarcs{G}{v}{v}$ and increases each $c(v')$ with $v'\neq v$ by $\numarcs{G}{v}{v'}$. The Burning algorithm presented in Lemma \ref{CFG:Dah} remains valid for Eulerian digraphs with loops.

As pointed above, the CFG on a digraph possibly having loops is very close to the CFG on the digraph where the loops are removed. For a digraph $G$ we denote by $\deleteloops{G}$ the digraph $G$ in which all loops are removed, and denote by $\numloops{G}$ the number of loops of $G$. Regarding undirected graphs, the influence of a loop is the same as a directed loop (it also ``freezes'' one chip). For two arcs $e$ and $e'$ of $G$, $e$ is \emph{\reverse} of $e'$ if $\tail{e}=\head{e'}$ and $\head{e}=\tail{e'}$. An undirected graph will be considered as an Eulerian digraph by replacing each edge $e$, that is not a loop, by two {\reverse} arcs $e'$ and $e''$ that have the same endpoints as $e$, and each undirected loop $e$ by exactly one directed loop $e'$ that has the same endpoint as $e$.

Theorem \ref{CFG:sink-independence} is generalized to the class of Eulerian digraphs possibly having loops with the following lemma. From now on, we will always consider an arbitrary fixed sink denoted $s$, therefore we won't use anymore the notations $*^{s}$ and $\circ^{s}$, but simply denote $c : V\backslash \set{s} \to \mathbb N$ a configuration.

\begin{lem}
\label{loop reduction}
Let $G=(V,A)$ be an Eulerian digraph with global sink $s$. Let $\mathcal{C}$ and $\overline{\mathcal{C}}$ be the sets of recurrent configurations of $G$ and $\deleteloops{G}$ with respect to sink $s$, respectively. For each configuration $c\in \overline{\mathcal{C}}$, let $\mu(c):V\backslash\set{s}\to \mathbb{N}$ be given by $\mu(c)(v)=c(v)+\numarcs{G}{v}{v}$ for any $v\neq s$. Then $\mu$ is a bijection from $\overline{\mathcal{C}}$ to $\mathcal{C}$. Moreover, $\underset{v\neq s}{\sum}c(v)-\underset{v\neq s}{\sum}\mu(c)(v)=-\underset{v\neq s}{\sum}\numarcs{G}{v}{v}$ for any $c\in \overline{\mathcal{C}}$.
\end{lem}

\begin{figure}
\centering
\subfloat[$\deleteloops{G}$]{\label{fig:im15}\includegraphics[bb=0 0 260 110,width=1.88in,height=0.791in,keepaspectratio]{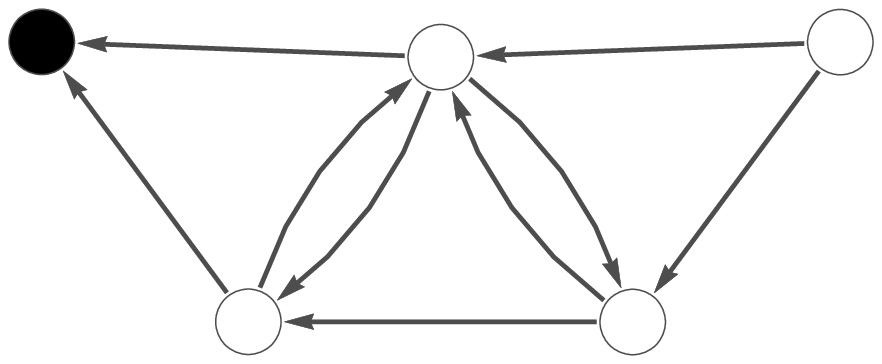}}\quad \quad
\subfloat[$G$]{\includegraphics[bb=0 0 260 140,width=1.88in,height=1.01in,keepaspectratio]{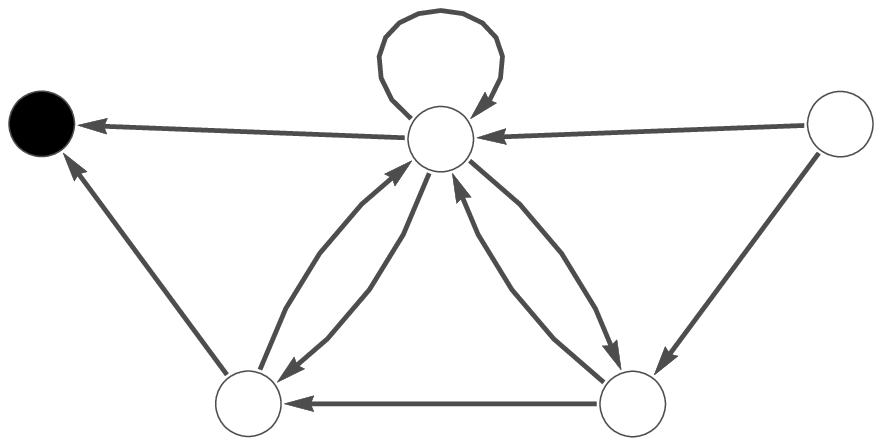}}\\
\subfloat[$c \in \overline{\mathcal C}$]{\label{fig:im16}\includegraphics[bb=0 0 234 104,width=1.88in,height=0.832in,keepaspectratio]{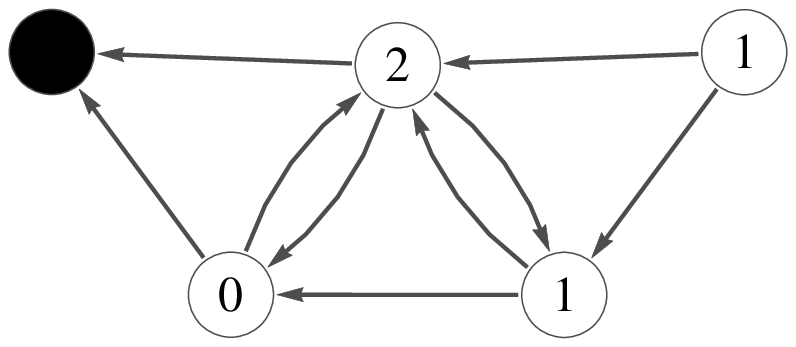}}\quad \quad
\subfloat[$\psi(c) \in \mathcal C$]{\label{fig:im17}\includegraphics[bb=0 0 260 140,width=1.88in,height=1.01in,keepaspectratio]{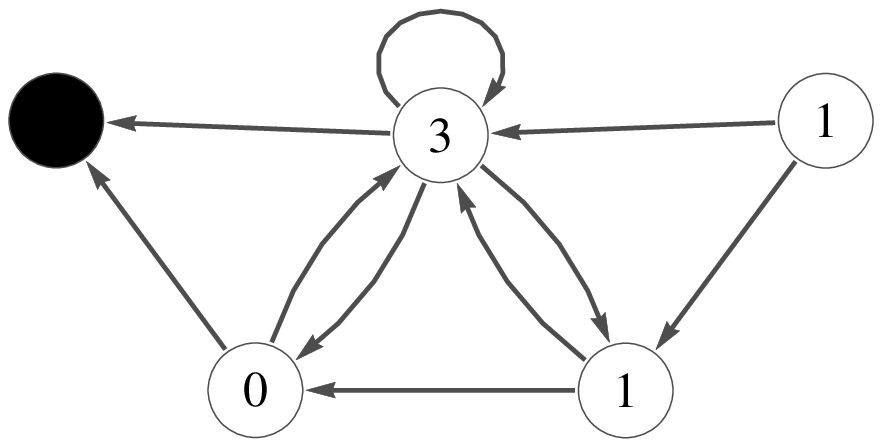}}
\caption{Relation between $\mathcal{C}$ and $\overline{\mathcal{C}}$}
\label{fig:im15161718}
\end{figure}

\noindent This lemma can be proved easily by using the definition of recurrent configuration with the observation that if a configuration $c$ is recurrent with respect to $G$ then $c(v)\geq \numarcs{G}{v}{v}$ for any $v\neq s$. An illustration of Lemma \ref{loop reduction} is given in Figure \ref{fig:im15161718}.

In the rest of this section, we work with an Eulerian digraph $G=(V,A)$ possibly having loops and an arbitrary but fixed vertex $s$ of $G$ that plays the role of sink for the game. We now introduce the partial Tutte polynomial generalization, which is defined as the generating function of the set of recurrent configurations. The generating function is based on the concept of {\em level} of recurrent configurations, which corresponds to the previously defined $\textbf{\texttt{sum}}$ normalized according to the smallest level of a recurrent configuration. For an Eulerian digraph $G$, let
$$\totalchipmin{G} \text{ denote the minimum of } \confdegree{\deleteloops{G}}{s}{c}=\outdegree{\deleteloops{G}}{s}+\underset{v\neq s}{\sum}c(v)$$
over all recurrent configurations $c$ of $\deleteloops{G}$ with respect to sink $s$. Theorem \ref{CFG:sink-independence} implies that $\totalchipmin{G}$ is independent of the choice of $s$. It follows from \cite{PePh13} that the problem of finding $\totalchipmin{G}$ is NP-hard for Eulerian digraphs, and as a consequence the Tutte-like polynomial we present is also NP-hard to compute. In addition, when $G$ is undirected (and defined as a digraph {\em i.e.}, each edge is represented by two {\reverse} arcs), the number $\totalchipmin{G}$ has an exact formula, namely $\totalchipmin{G}=\frac{|A(\deleteloops{G})|}{2}$. For a recurrent configuration $c$ of $G$ with respect to sink $s$ we define 
$$\level{G}{c}=\confdegree{G}{s}{c}-\totalchipmin{G}.$$
This is a generalization of the level that was defined in \cite{Lop97}, because we recover the latter when $G$ is undirected. Let $\mathcal{C}$ denote the set of all recurrent configurations of $G$ with respect to sink $s$, the generating function of $\mathcal{C}$ is given by
$$\Tut{G}{y}=\underset{c\in \mathcal{C}}{\sum} y^{\level{G}{c}}$$
and we claim that it is a natural generalization of the partial Tutte polynomial, for the class of Eulerian digraphs. First, it follows from Theorem \ref{CFG:sink-independence} that $\Tut{G}{y}$ is independent of the choice of $s$, thus is characteristic of the support graph $G$ itself. We are going to present in this section a number of properties of $\Tut{G}{y}$ that can be considered as the generalizations of those of the Tutte polynomial in one variable, namely $T_G(1,y)$, that is defined on undirected graphs. The most interesting and new feature is that when $G$ is an undirected graph we get back to the well-known Tutte polynomial. This fact is straightforward to notice.

\begin{itemize}
  \item $\Tut{G}{y}=T_G(1,y)$ if $G$ is an undirected graph.
  \item $\Tut{G}{1}$ counts the number of oriented spanning tree of $G$ rooted at $s$ \cite{HLMPPW08}. It generalizes the evaluation $T_G(1,1)$ that counts the number of spanning tree of an undirected graph.
  \item $\Tut{G}{0}$ counts the number of maximum acyclic arc sets with exactly one sink $s$ \cite{PePh13}. Therefore $\Tut{G}{0}$ is a natural generalization of $T_G(1,0)$ that counts the number of acyclic orientations with a fixed source of an undirected graph.
\end{itemize}  

\begin{ques}
What does $\Tut{G}{2}$ count?
\end{ques}

Those evaluations set up a promising ground for further investigations, but it is definitely not trivial to find out the objects counted by evaluations of graph polynomials. We are now going to present the extension to $\Tut{G}{y}$ of four known recursive formulas for the Tutte polynomial in the undirected case. We will need the two following simple lemmas.

\begin{notation}
For a function $f:X\to Y$ and a subset $X'\subseteq X$ we denote by $f_{|X'}$ the restriction of $f$ to $X'$.
\end{notation}

For a subset $A'$ of $A$ let $\deletearc{G}{A'}$ denote the graph $(V,A\backslash A')$. We write $\deletearc{G}{e}$ for $\deletearc{G}{\set{e}}$ if $e$ is an arc of $G$. For an arc $e$ of $G$ with two endpoints $v$ and $v'$ let $\contract{G}{e}$ denote the digraph that is made from $G$ by removing $e$ from $G$, replacing $v$ and $v'$ by a new single vertex $v''$, and for each remaining arc $e'$ if the head (resp. tail) of $e'$ in $G$ is $v$ or $v'$ then the head (resp. tail) of $e'$ in $\contract{G}{e}$ is $v''$. This procedure is called \emph{arc contraction}. See Figure \ref{fig:im1920} for an illustration of the arc contraction. Note that $\contract{G}{e}$ is still Eulerian.

An analogue of arc contraction may also be defined for vertices. For a subset $W$ of $V$, let $\vertexcontract{G}{W}$ denote the digraph constructed from $G$ by replacing all vertices in $W$ by a new vertex $w'$, and each arc $e\in A$ such that $\tail{e}\in W$ (resp. $\head{e}\in W$) in $G$ by $\tail{e}=w'$ (resp. $\head{e}=w'$) in $\vertexcontract{G}{W}$. The following is originally due to Merino L\'opez \cite{Lop97}.

\begin{figure}
\centering
\subfloat[$G$]{\label{fig:im19}\includegraphics[bb=0 0 382 171,width=1.88in,height=0.84in,keepaspectratio]{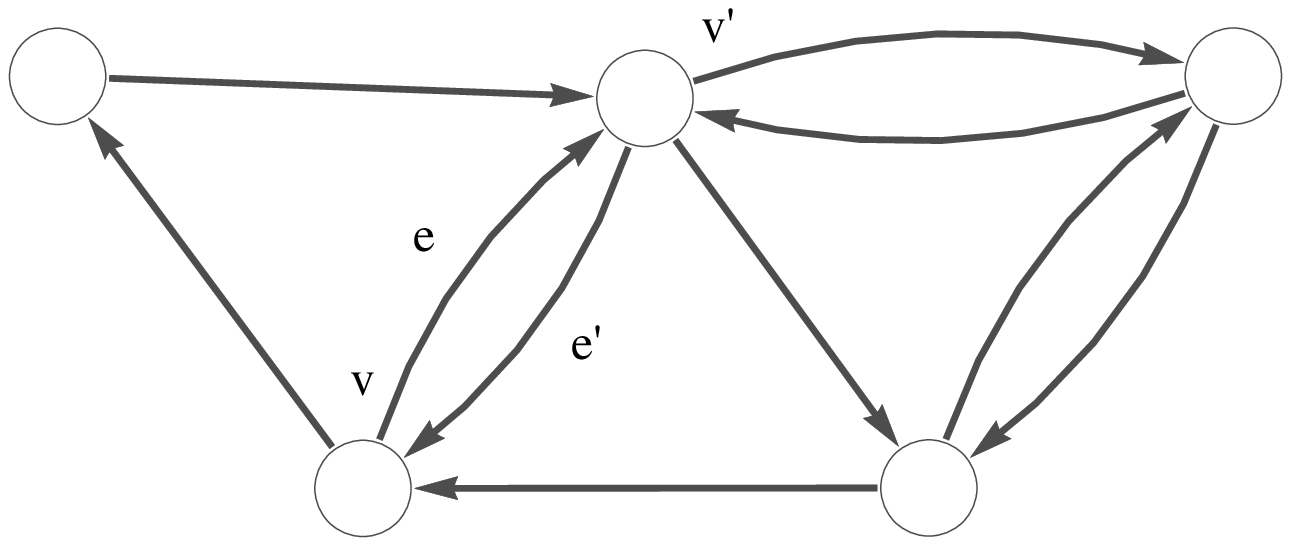}}\quad \quad
\subfloat[$\contract{G}{e}$]{\label{fig:im20}\includegraphics[bb=0 0 340 187,width=1.88in,height=1.03in,keepaspectratio]{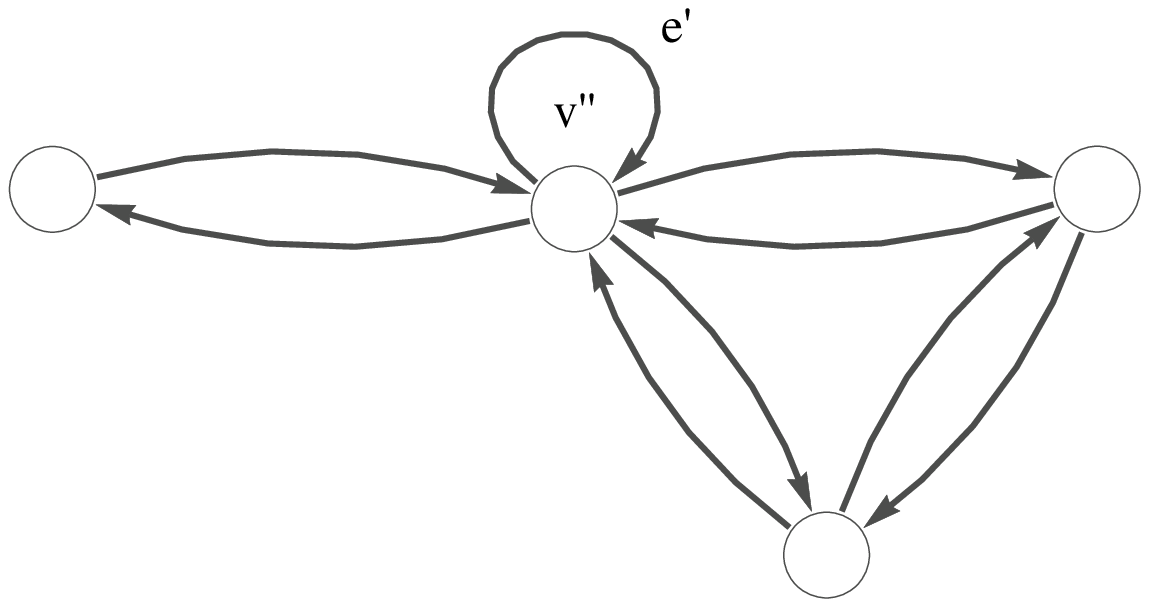}}
\caption{Arc contraction}
\label{fig:im1920}
\end{figure}

\begin{lem}
\label{recurrent:contraction}
Let $W$ be a non-empty subset of the set of out-neighbors of $s$, \emph{i.e.}, for every $v\in W$ we have $v\neq s$ and $(s,v) \in A$. Let $s'$ be the new vertex in $\vertexcontract{G}{W\cup \set{s}}$ resulting from replacing the set of vertices $W\cup \set{s}$. For any $c\in \mathcal{C}$, if $c(v)\geq \outdegree{G}{v}-\numarcs{G}{s}{v}$ for all $v\in W$, then $c_{|V\backslash (W\cup \set{s})}$ is a recurrent configuration of $\vertexcontract{G}{W\cup \set{s}}$ with respect to sink $s'$. Conversely, if $c'$ is a recurrent configuration of $\vertexcontract{G}{W\cup \set{s}}$ with respect to sink $s'$ then every configuration $c:V\backslash\set{s}\to \mathbb{N}$, satisfying $c(v)=c'(v)$ for all $v\in V\backslash(W\cup \set{s})$ and $\outdegree{G}{v} > c(v)\geq \outdegree{G}{v}-\numarcs{G}{s}{v}$ for all $v\in W$, is in $\mathcal{C}$.
\end{lem}

\begin{proof}
This proof is straightforward, we use Lemma \ref{CFG:Dah} (Burning algorithm) and the hypothesis that a configuration is recurrent, thus it admits a firing sequence, in order to construct a firing sequence for the considered configuration, which proves that it is recurrent (again by Lemma \ref{CFG:Dah}).

We denote the vertices in $W$ by $w_1,w_2,\cdots,w_q$. Let the configuration $\beta:V\backslash\set{s}\to \mathbb{N}$ be given by $\beta(v)=\numarcs{G}{s}{v}$ for any $v\in V\backslash\set{s}$. The condition $c(w_i)\geq \outdegree{G}{s}-\numarcs{G}{s}{w_i}$ for any $i$ implies that $w_i$ is firable in $c+\beta$ for any $i$. It follows from Lemma \ref{CFG:convergence} and Lemma \ref{CFG:Dah} that there is a firing sequence $\mathfrak{f}=(v_1,v_2,\cdots,v_k)$ of $c+\beta$ in $G$ such that $c+\beta\overset{\mathfrak{f}}{\to} c$, $v_i\neq s$ for any $i$, each vertex of $G$ distinct from $s$ occurs exactly once in $\mathfrak{f}$, and $v_i=w_i$ for any $i\in [1..q]$. Let $d$ be such that $c+\beta\overset{w_1,w_2,\cdots,w_p}{\longrightarrow} d$. Let $\beta':V\backslash (W\cup \set{s})\to \mathbb{N}$ be given by $\beta'(v)=\numarcs{\vertexcontract{G}{W\cup \set{s}}}{s'}{v}$ for any $v\in V\backslash (W\cup \set{s})$. Clearly, we have $d_{|V\backslash (W\cup \set{s})}=c_{|V\backslash(W\cup \set{s})}+\beta'$.  Since $\mathfrak{f}'=(v_{p+1},v_{p+2},\cdots,v_k)$ is a firing sequence of $d$, $\mathfrak{f}'$ is also a firing sequence of $c_{|V\backslash(W\cup \set{s})}+\beta'$ in $\vertexcontract{G}{W\cup \set{s}}$. It follows from Lemma \ref{CFG:Dah} that $c_{|V\backslash (W\cup \set{s})}$ is a recurrent configuration of $\vertexcontract{G}{W\cup \set{s}}$ with respect to sink $s'$.

For the converse statement let $\mathfrak{g}=({v'}_1,{v'}_2,\cdots,{v'}_p)$ be a firing sequence of $c'$ such that $c'+\beta'\overset{\mathfrak{g}}{\to}c'$ in $\vertexcontract{G}{W\cup \set{s}}$, then $v'_i\not \in W\cup \set{s}$ for any $i$, and each vertex of $G$ not in $W\cup \set{s}$ occurs exactly once in $\mathfrak{g}$. Let $c''$ be such that $c+\beta\overset{w_1,w_2,\cdots,w_q}{\longrightarrow}c''$ in $G$. Clearly, $c''_{|V\backslash(W\cup \set{s})}=c'+\beta'$, therefore $(w_1,w_2,\cdots,w_q,{v'}_1,{v'}_2,\cdots,{v'}_p)$ is a firing sequence of $c+\beta$ in $G$. It follows that $c\in \mathcal{C}$.
\end{proof}

\begin{lem}
\label{generatingfunction:deletion}
Let $e$ and $e'$ be two {\reverse} arcs of $G$ such that they are not loops and $\tail{e}=s$. Let $H$ denote $\deletearc{G}{\set{e,e'}}$ and $w$ denote $\head{e}$. If $H$ is connected then $\set{c \in \mathcal{C}:c(w)<\outdegree{G}{w}-1}$ is the set of all recurrent configurations of $H$ with respect to sink $s$.
\end{lem}

\begin{proof}
We prove a double inclusion, using again both directions of Lemma \ref{CFG:Dah} (Burning algorithm).

Let $\beta:V\backslash\set{s}\to \mathbb{N}$ be given as in Lemma \ref{recurrent:contraction} and $\overline{\beta}:V\backslash\set{s}\to \mathbb{N}$ be given by $\beta(v)=\numarcs{H}{s}{v}$ for any $v \in V\backslash\set{s}$. We have $\beta(v)=\overline{\beta}(v)$ for any $v\neq w$, and $\beta(w)-\overline{\beta}(w)=1$. Let $c\in \mathcal{C}$ such that $c(w)<\outdegree{G}{w}-1$. We have to prove that $c$ is also a recurrent configuration of $H$ with respect to sink $s$. Let $\mathfrak{f}=(v_1,v_2,\cdots,v_k)$ be a firing sequence of $c$ in $G$ such that $v_i\neq s$ for any $i$, $c+\beta\overset{\mathfrak{f}}{\to} c$, and each vertex of $G$ distinct from $s$ occurs exactly once in $\mathfrak{f}$. We will show that $\mathfrak{f}$ is a firing sequence of $c+\overline{\beta}$ in $H$. Let $j$ be such that $v_j=w$. Clearly, $(v_1,v_2,\cdots,v_{j-1})$ is a firing sequence of $c+\beta$ and $c+\overline{\beta}$ in $G$ and $H$, respectively. Let $c'$ be such that $c+\beta\overset{v_1,v_2,\cdots,v_{j-1}}{\longrightarrow}c'$ in $G$ and $d'$ be such that $c+\overline{\beta}\overset{v_1,v_2,\cdots,v_{j-1}}{\longrightarrow} d'$ in $H$. It follows from $\beta(w)-\overline{\beta}(w)=1$ that $c'(w)-d'(w)=1$. To prove that $\mathfrak{f}$ is a firing sequence of $c+\overline{\beta}$ in $H$ it suffices to show that $v_j$ is firable in $d'$ with respect to $H$. Since $v_j$ is firable in $c'$ with respect to $G$, we have $c'(w)\geq \outdegree{G}{w}$, therefore $d'(w)\geq \outdegree{G}{w}-1$. It follows that $w$ is firable in $d'$ with respect to $H$. This implies that $\mathfrak{f}$ is also a firing sequence of $c+\overline{\beta}$ with respect to $H$. By Lemma \ref{CFG:Dah}, $c$ is a recurrent configuration of $H$ with respect to sink $s$.

For the converse, let $d$ be a recurrent configuration of $H$ with respect to sink $s$. Let $\mathfrak{f}=(v'_1,v'_2,\cdots,v'_p)$ be a firing sequence of $d+\overline{\beta}$ in $H$ such that $v'_i\neq s$ for any $i$, $d+\overline{\beta}\overset{\mathfrak{f}'}{\to}d$ in $H$, and each vertex of $H$ distinct from $s$ occurs exactly once in $\mathfrak{f}'$. We have $d(w)\leq \outdegree{H}{w}-1<\outdegree{G}{w}-1$. By similar arguments as above, $\mathfrak{f}'$ is also a firing sequence of $d+\beta$ in $G$, therefore $d$ is a recurrent configuration of $G$ with respect to sink $s$.
\end{proof}

First, the Tutte polynomial has the recursive formula $T_G(1,y)=y\,T_{\deletearc{G}{e}}(1,y)$ if $e$ is a loop. We have the following generalization.

\begin{prop}
\label{loop recursion}
If $e$ is a loop then $\Tut{G}{y}=y\,\Tut{\deletearc{G}{e}}{y}$.
\end{prop}

\begin{proof}
Let $s$ denote $\tail{e}$. Let $\mathcal{C}$ be the set of all recurrent configurations of $G$ with sink $s$. Clearly, $\mathcal{C}$ is also the set of all recurrent configuartions of $\deletearc{G}{e}$ with sink $s$. Since $\outdegree{G}{s}-\outdegree{H}{s}=1$, for any $c \in \mathcal{C}$ we have $\level{G}{c}-\level{H}{c}=1$. This implies that $\Tut{G}{y}=y\Tut{\deletearc{G}{e}}{y}$.
\end{proof}

Second, in order to generalize the recursive formula $T_G(1,y)=T_{\contract{G}{e}}(1,y)$ if $e$ is a bridge, we generalize the notion of bridge to directed graphs with the following.

\begin{defi}
An arc $b$ in $G$ is called \emph{bridge} if $\deletearc{G}{b}$ is not strongly connected.
\end{defi}


The next lemma is used in the proofs of subsequent propositions, aimed at showing that this definition of bridge is a natural generalization of the same notion on undirected graphs.


\begin{lem}
\label{bridge:separation}
Let $b$ be a bridge of $G$. Then there is a subset $X$ of $V$ such that $\set{b}=\set{e\in A: \tail{e}\in X,\head{e}\not \in X}$. Moreover, there is an arc $b'$ in $G$ such that $\set{b'}=\set{e\in A: \tail{e}\not \in X, \head{e}\in X}$.
\end{lem}

\begin{proof}
Let $X$ be the set of all vertices $v$ of $G$ such that there is a path in $\deletearc{G}{b}$ from $\tail{b}$ to $v$. First, we show that $\head{b}\not \in X$. For a contradiction we assume that $\head{b}\in X$. This implies that there is a path $P$ in $\deletearc{G}{b}$ from $\tail{b}$ to $\head{b}$. Since $G$ is strongly connected and $\deletearc{G}{e}$ is not strongly connected, there exists two vertices $v_1,v_2$ of $G$ such that every path in $G$ from $v_1$ to $v_2$ must contain $b$. Let $Q$ be a path in $G$ from $v_1$ to $v_2$. We can assume that $b$ occurs exactly once in $Q$. Let $Q_1$ be a subpath of $Q$ from $v_1$ to $\tail{b}$, and $Q_2$ be a subpath of $Q_2$ from $\head{b}$ to $v_2$. Then, $(Q_1,P,Q_2)$ is a path in $G$ from $v_1$ to $v_2$ that does not contain $b$, a contradiction.

Let $B$ denote $\set{e\in A: \tail{e}\in X, \head{e}\not \in X}$. The above claim implies that $b \in B$. If there is an arc $e$ of $G$ such that $e\neq b$ and $e\in B$. It follows from the definition of $X$ that $\head{e}\in X$, a contradiction. Therefore $b$ is the unique element in $B$. The first statement follows.

Since $G$ in an Eulerian digraph, for every subset $X$ of $V$ there are as many arcs from $X$ to $V\backslash X$ as from $V\backslash X$ to $X$. The second claim follows.
\end{proof}

\begin{figure}
\centering
\includegraphics[bb=0 0 346 260,width=2.81in,height=1.77in]{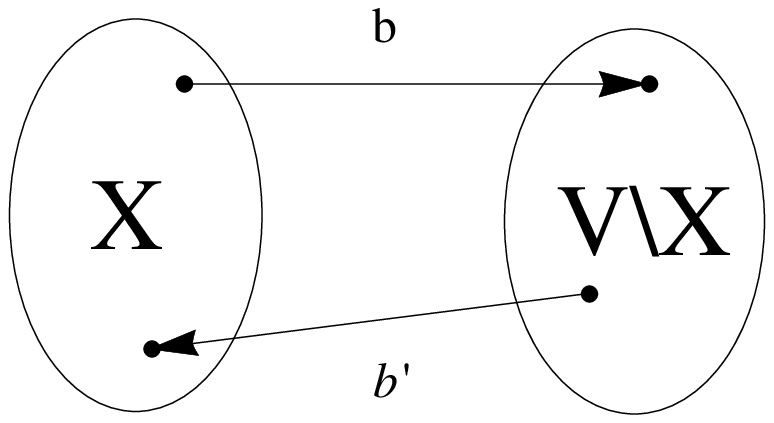}
\caption{Bridge}
\label{fig:im21}
\end{figure}

See Figure \ref{fig:im21} for the illustration of Lemma \ref{bridge:separation}. Note that the set $X$ satisfying the condition of the lemma may not be unique. The following convinces that this definition of bridge is a natural extension of the known notion on undirected graphs.

\begin{prop}
If $G$ is an undirected graph (seen as a directed graph), then an arc $e$ is a bridge of $G$ iff there is a {\reverse} arc $e'$ of $e$ in $G$ and $\deletearc{G}{\set{e,e'}}$ is not connected.
\end{prop}

\begin{proof}
$\Rightarrow$ Let $X$ be a subset of $V$ that satisfies the condition in Lemma \ref{bridge:separation}. Since $G$ is undirected, there is a {\reverse} arc $e'$ of $e$ in $G$. Clearly, $e'\in \set{e\in E: \tail{e}\not \in X,\head{e}\in X}$. Lemma \ref{bridge:separation} implies that $\set{e'}=\set{e\in E: \tail{e}\not \in X,\head{e}\in X}$. It follows that $\deletearc{G}{\set{e,e'}}$ contains no arc from $X$ to $V\backslash X$ and vice versa. Therefore $\deletearc{G}{\set{e,e'}}$ is not connected.

$\Leftarrow$ Since $G$ is connected and $\deletearc{G}{\set{e,e'}}$ is not connected, $\tail{e}$ and $\head{e}$ are in different connected components of $\deletearc{G}{\set{e,e'}}$. Let $X$ and $Y$ be two connected components of $\deletearc{G}{\set{e,e'}}$ such that $\tail{e}\in X$ and $\head{e}\in Y$. Let $v\in X$ and $v'\in Y$. Since there is no arc in $\deletearc{G}{e}$ from a vertex in $X$ to a vertex in $Y$, there is no path in $\deletearc{G}{e}$ from $v$ to $v'$. This implies that $\deletearc{G}{e}$ is not strongly connected. Therefore $e$ is a bridge.
\end{proof}

The second relation, extending the recursive formula on undirected graphs $T_G(1,y)=T_{\contract{G}{e}}(1,y)$ if $e$ is a bridge, is split into the two following propositions, depending on whether the bridge has a reverse arc.

\begin{prop}
\label{non-reverse bridge recursion}
Let $e$ be a bridge of $G$ such that it does not have a {\reverse} arc. Then $\Tut{G}{y}=\Tut{\contract{G}{e}}{y}$.
\end{prop}

\begin{proof}
We construct a bijection from the set of recurrent configurations of $\contract{G}{e}$ to the set of recurrent configurations of $G$ that preserves the level. We prove two intermediate claims, and the result follows.

Let $s$ denote $\tail{e}$ and $t$ denote $\head{e}$. Let $\mathcal{C}$ be the set of all recurrent configurations of $G$ with respect to the sink $s$. We claim that for any $c \in \mathcal{C}$ we have $c(t)=\outdegree{G}{t}-1$. For a contradiction we assume that $c(t)<\outdegree{G}{t}-1$. Let $X$ be a subset of $V$ that satisfies the condition of Lemma \ref{bridge:separation}. Let $\beta:V\backslash\set{s}\to \mathbb{N}$ be given by $\beta(v)=\numarcs{G}{s}{v}$ for any $v\in V\backslash\set{s}$. The choice of $X$ straightforwardly implies that $\beta(t)=1$, and $\beta(v)=0$  for any $v\in V\backslash (X\cup \set{t})$. Let $\mathfrak{f}=(v_1,v_2,\cdots,v_k)$ be a firing sequence of $c+\beta$ such that $v_i\neq s$ for any $i$, $c+\beta\overset{\mathfrak{f}}{\to} c$, and each vertex $v$ of $G$ distinct from $s$ occurs exactly once in the sequence. Since $c(t)<\outdegree{G}{t}-1$, there is no firable vertex of $c+\beta$ in $V\backslash X$. This implies that $v_1\in X$. Let $j$ be the smallest index such that $v_j\in X$ and $v_{j+1}\not\in X$, and $c'$ be the configuration reach after the $j$ first vertices have been fired, that is, such that $c\overset{v_1,v_2,\cdots,v_j}{\longrightarrow}c'$. Since $v_{j+1}$ is not firable in $c+\beta$ and firable in $c'$, there is at least one vertex $v_p\in \set{v_1,v_2,\cdots,v_j}$ that gives chips to $v_{j+1}$ when it is fired. It follows that there is at least one arc $e'$ of $G$ such that $\tail{e'}=v_p$ and $\head{e'}=v_{j+1}$. Clearly, $e'\neq e$ and $e'\in \set{e\in A:\tail{e}\in X,\head{e}\not \in X}$, a contradiction to Lemma \ref{bridge:separation}.

Let $H$ denote $\contract{G}{e}$, let $s'$ denote the vertex of $H$ resulting from replacing $s$ and $t$ in $\contract{G}{e}$, and let $\mathcal{C}'$ denote the set of all recurrent configurations of $H$ with the sink $s'$. We claim that $\totalchipmin{G}=\totalchipmin{H}$. By Lemma \ref{loop reduction} we have $\totalchipmin{G}=\min\set{\confdegree{G}{s}{c}-\numloops{G}:c\in \mathcal{C}}$ and $\totalchipmin{H}=\min\set{\confdegree{H}{s'}{c}-\numloops{H}:c\in \mathcal{C}'}$, where $\numloops{G}$ and $\numloops{H}$ are the numbers of loops of $G$ and $H$, respectively. It follows from the above claim and Lemma \ref{recurrent:contraction} that the map $\mu:\mathcal{C}'\to \mathcal{C}$, defined by $\mu(c)(v)=c(v)$ if $v\neq t$, and $\mu(c)(t)=\outdegree{G}{t}-1$, is a bijection. Therefore $\min\set{\underset{v\neq s}{\sum}c(v):c\in \mathcal{C}}-\min\set{\underset{v\neq s}{\sum} c(v):c\in \mathcal{C}'}=\outdegree{G}{t}-1$. Note that $\outdegree{H}{s'}=\outdegree{G}{s}+\outdegree{G}{t}-1$. Finally, since $e$ does not have a {\reverse} arc, we have $\numloops{G}=\numloops{H}$, and the claim follows the fact that $\confdegree{G}{s}{c}=\outdegree{G}{s}+\sum \limits_{v\neq s} c(v)$.

We can conclude the proof: for any $c\in \mathcal{C}'$ we have $\level{G}{\mu(c)}=\outdegree{G}{s}+\underset{v\neq s}{\sum} \mu(c)(v)-\totalchipmin{G}=\outdegree{G}{s}+\underset{v\neq s'}{\sum} c(v)+\outdegree{G}{t}-1-\totalchipmin{H}=\outdegree{H}{s'}+\underset{v\neq s'}{\sum} c(v)-\totalchipmin{H}=\level{H}{c}$. This implies $\Tut{G}{y}=\Tut{H}{y}$.
\end{proof}

\begin{prop}
\label{reverse-bridge recursion}
Let $e$ be a bridge of $G$ such that it has a {\reverse} arc $e'$, and let $H$ denote $\contract{G}{e}$.\\
Then $\Tut{G}{y}=\frac{1}{y}\Tut{H}{y}$ and $\Tut{G}{y}=\Tut{\deletearc{H}{e'}}{y}$.
\end{prop}

As shown on Figure \ref{fig:im21}, deleting $e'$ in $H$ corresponds to erasing the loop created by the contraction of $e$.

\begin{proof}
It follows from Lemma \ref{bridge:separation} that $e'$ is the unique {\reverse} arc of $e$. Let $s'$ be the new vertex in $\contract{G}{e}$ resulting from replacing the two endpoints of $e$. Let $\mathcal{C}$ and $\mathcal{C}'$ be the sets of all recurrent configurations of $G$ and $H$ with respect to the sinks $s$ and $s'$, respectively. The following can be proved by similar arguments as used in the proof of Proposition \ref{non-reverse bridge recursion} with the notice that $e'$ is a loop in $\contract{G}{e}$.
\begin{itemize}
  \item $\numloops{H}=\numloops{G}+1$.
  \item the map $\mu:\mathcal{C}'\to \mathcal{C}$, defined by $\mu(c)(v)=c(v)$ if $v\neq t$, and $\mu(c)(t)=\outdegree{G}{t}-1$, is a bijection, where $t$ denotes $\head{e}$.
  \item $\totalchipmin{H}=\totalchipmin{G}-1$.
  \item for any $c\in \mathcal{C}'$ $\level{H}{c}=\level{G}{\mu(c)}+1$.
\end{itemize} 
The assertions above imply that $\Tut{G}{y}=\frac{1}{y}\Tut{H}{y}$. Since $e'$ is a loop in $H$, it follows from Proposition \ref{loop recursion} that  $\Tut{G}{y}=\frac{1}{y}\Tut{H}{y}=\frac{1}{y}\, y\,\Tut{\deletearc{H}{e'}}{y}=\Tut{\deletearc{H}{e'}}{y}$.
\end{proof}

Third, the recursive formula $T_G(1,y)=T_{\deletearc{G}{e}}(1,y)+T_{\contract{G}{e}}(1,y)$ if $e$ is neither a loop nor a bridge has the following generalization.

\begin{prop}
\label{neither loop nor bridge recursion}
Let $e$ be an arc of $G$ such that $e$ is neither a loop nor a bridge, and $e$ has a {\reverse} arc $e'$. Then $\Tut{G}{y}=y^{1+\totalchipmin{\deletearc{G}{\set{e,e'}}}-\totalchipmin{G}}\,\Tut{\deletearc{G}{\set{e,e'}}}{y}+y^{\totalchipmin{H}-\totalchipmin{G}}\,\Tut{H}{y}$, where $H$ denotes $\contract{G}{e}$. Moreover, if $G$ is undirected then $\Tut{G}{y}=\Tut{\deletearc{G}{\set{e,e'}}}{y}+y^{-\numarcs{G}{\tail{e}}{\head{e}}+1}\,\Tut{\deletearc{H}{e'}}{y}$.
\end{prop}

In this formula, we reduce $\Tut{G}{y}$ to the sum of the polynomial for $G$ on which both $e$ and its reverse arc $e'$ are removed (corresponding to the bridge deletion of the undirected case, see Proposition \ref{reverse-bridge recursion}) and the polynomial for $G$ on which $e$ is contracted. The terms $y^\alpha$ are used for re-normalizing according to the definition of level.

\begin{proof}
We first give names to useful elements, and then prove both statements of the result one after the other. Let $s$ and $t$ denote $\tail{e}$ and $\head{e}$, respectively. Since $e$ is neither a loop nor a bridge, $\deletearc{G}{\set{e,e'}}$ is connected. Let $\mathcal{C}_1$ be the set of all recurrent configurations $c$ of $G$ with sink $s$ such that $c(t)=\outdegree{G}{t}-1$, and let $\mathcal{C}_2$ be the set of all recurrent configurations $c$ of $G$ with sink $s$ such that $c(t)<\outdegree{G}{t}-1$. We have $\mathcal C=\mathcal C_1 \cup C_2$, and we will see that each element of this partition corresponds to one of the two terms of the sum. Let $s'$ denote the vertex of $H$ resulting from replacing $s$ and $t$ in $G$. Let $\mathcal{D}$ be the set of all recurrent configurations of $H$ with sink $s'$. 

\textbf{First statement.} We begin with $\underset{c\in \mathcal{C}_1}{\sum} z^{\level{G}{c}}$, corresponding to the second term of the sum. It follows from Lemma \ref{recurrent:contraction} that the map $\mu:\mathcal{D}\to \mathcal{C}_1$, defined by $\mu(c)(v)=c(v)$ if $v\neq t$, and $\mu(c)(t)=\outdegree{G}{t}-1$, is bijective. For any $c\in \mathcal{C}_1$ we have $\level{G}{c}=\outdegree{G}{s}+\underset{v\neq s}{\sum}c(v)-\totalchipmin{G}=\outdegree{G}{s}+\outdegree{G}{t}-1+\underset{v\not \in \set{s,t}}{\sum} c(v)-\totalchipmin{H}+\totalchipmin{H}-\totalchipmin{G}=\outdegree{H}{s'}+\underset{v\not \in \set{s,t}}{\sum}c(v)-\totalchipmin{H}+\totalchipmin{H}-\totalchipmin{G}=\level{H}{\mu^{-1}(c)}+\totalchipmin{H}-\totalchipmin{G}$. This implies that $\underset{c\in \mathcal{C}_1}{\sum} z^{\level{G}{c}}=y^{\totalchipmin{H}-\totalchipmin{G}}\,\Tut{H}{y}$, which is the second term of the sum.

Regarding $\underset{c\in \mathcal{C}_2}{\sum} z^{\level{G}{c}}$, it follows from Lemma \ref{generatingfunction:deletion} that $\mathcal{C}_2$ is the set of all recurrent configurations of $\deletearc{G}{\set{e,e'}}$ with sink $s$. For any $c\in \mathcal{C}_2$ we have $\level{G}{c}=\outdegree{G}{s}+\underset{v\neq s}{\sum}c(v)-\totalchipmin{G}=1+\outdegree{\deletearc{G}{\set{e,e'}}}{s}+\underset{v\neq s}{\sum}c(v)-\totalchipmin{\deletearc{G}{\set{e,e'}}}+\totalchipmin{\deletearc{G}{\set{e,e'}}}-\totalchipmin{G}=\level{\deletearc{G}{\set{e,e'}}}{c}+1+\totalchipmin{\deletearc{G}{\set{e,e'}}}-\totalchipmin{G}$. This implies that $\underset{c\in \mathcal{C}_2}{\sum} y^{\level{G}{c}}=y^{1+\totalchipmin{\deletearc{G}{\set{e,e'}}}-\totalchipmin{G}}\,\Tut{\deletearc{G}{\set{e,e'}}}{y}$. Since $\Tut{G}{y}=\underset{c\in \mathcal{C}_1}{\sum}y^{\level{G}{c}}+\underset{c\in \mathcal{C}_2}{\sum}y^{\level{G}{c}}$, the first statement follows.

\textbf{Second statement.} $G$ is an undirected graph, so are $\deletearc{G}{\set{e,e'}}$ and $H$. Thus $1+\totalchipmin{\deletearc{G}{\set{e,e'}}}-\totalchipmin{G}=1+\frac{|A(\deleteloops{\deletearc{G}{\set{e,e'}}})|}{2}-\frac{|A(\deleteloops{G})|}{2}=0$. Since $e'$ is a loop in $H$, we have $y^{\totalchipmin{H}-\totalchipmin{G}}\,\Tut{H}{y}= y^{1+\totalchipmin{H}-\totalchipmin{G}}\,\Tut{\deletearc{H}{e'}}{y}$. The second statement is completed by showing that $\totalchipmin{H}-\totalchipmin{G}=-\numarcs{G}{s}{t}$. We have $\totalchipmin{H}-\totalchipmin{G}=\frac{|A(H)|-\numloops{H}}{2}-\frac{|A(G)|-\numloops{G}}{2}=\frac{|A(H)|-|A(G)|}{2}-\frac{\numloops{H}-\numloops{G}}{2}=-\frac{1}{2}-\frac{(2\numarcs{G}{s}{t}-1)}{2}=-\numarcs{G}{s}{t}$.
\end{proof}

Let us present a new formula that does not exist for the Tutte polynomial on undirected graphs. If $G$ is undirected, then it contains at least one arc that is a loop, or it satisfies the conditions of Proposition \ref{non-reverse bridge recursion}, Proposition \ref{reverse-bridge recursion} or Proposition \ref{neither loop nor bridge recursion}. In every case $\Tut{G}{y}$ can be defined by a recursive formula on smaller graphs. However, the digraph given in Figure \ref{fig:im22} is an example of Eulerian digraph that does not contain any such arc, therefore no recursive formula generalizing those of the classical Tutte polynomial can be applied. Neither of the recursive formulas in Proposition \ref{loop recursion}, Proposition \ref{non-reverse bridge recursion}, Proposition \ref{reverse-bridge recursion} and Propostion \ref{neither loop nor bridge recursion} is useful in this case. The following new recursive formula handles this case, in order to complete the recursive definitions of $\Tut{G}{y}$ on the class of general Eulerian digraphs. Note that its intuitive shape comes from the Mobius inversion formula that is stated as follows.

\begin{figure}
\centering
\includegraphics[bb=0 0 247 227,width=1.71in,height=1.57in,keepaspectratio]{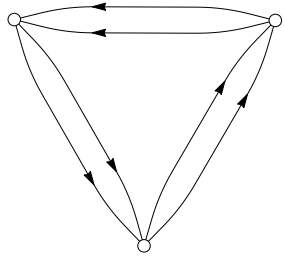}
\caption{An Eulerian digraph that does not satisfy any usual condition}
\label{fig:im22}
\end{figure}

\emph{Mobius inversion formula.} Let $X$ be a non-empty finite set and $f:2^{X}\to \mathbb{Z}$. We define $g:2^X\to \mathbb{Z}$ by $g(A):=\underset{A\subseteq Y}{\sum} f(Y)$. Then for every $A\in 2^X$ we have
$$
f(A)=\underset{A\subseteq Y}{\sum}(-1)^{|Y|-|A|} g(Y)
$$
\begin{prop}
\label{general recursive formula}
Let $G$ be an Eulerian digraph, $s$ be a vertex of $G$, and $N$ be the set of all out-neighbors of $s$. Then
$$\Tut{G}{y}=\sum_{\substack{W\subseteq N\\ W\neq \emptyset}} (-1)^{|W|+1}y^{\totalchipmin{\vertexcontract{G}{W\cup \set{s}}}-\totalchipmin{G}-\numarcs{G}{s}{W}} \frac{1}{(1-y)^{|W|}}\underset{v\in W}{\prod}\left(1-y^{\numarcs{G}{s}{v}}\right)\Tut{\vertexcontract{G}{W\cup \set{s}}}{y}$$
where $\numarcs{G}{s}{W}$ denotes the number of arcs $e$ of $G$ such that $\tail{e}=s$ and $\head{e}\in W$.
\end{prop}

\noindent Note that the number of vertices of the digraph $\vertexcontract{G}{W\cup \set{s}}$ is strictly smaller than $G$. Moreover the digraph $\vertexcontract{G}{W\cup \set{s}}$ is likely to have more loops than $G$, hence we could apply Proposition \ref{loop recursion} to remove the loops in $\vertexcontract{G}{W\cup \set{s}}$.
\begin{proof}
Let $\mathcal{C}$ be the set of all recurrent configurations of $G$ with sink $s$. For each $c\in \mathcal{C}$, let $\support{c}$ be the set of out-neighbors of $s$ that, from the configuration $c$, become firable when $s$ is fired, formally $\support{c}:=\set{v\in N:c(v)\geq \outdegree{G}{v}-\numarcs{G}{s}{v}}$. We define
$$P_W(y)=\sum_{\substack{c\in \mathcal C\\W \subseteq \support{c}}} y^{\level{G}{c}}$$
so that $\Tut{G}{y}=P_\emptyset(y)$. We will give thereafter a closed formula for $P_W(y)$, which is not interesting if $W=\emptyset$. In order to overcome this issue, let us express $P_\emptyset(y)$ in terms of $P_W(y)$ for $W\neq \emptyset$, using the Mobius inversion formula.

We define $Q_W(y)=\sum \limits_{\substack{c \in \mathcal C\\ W=\support{c}}} y^{\level{G}{c}}$ so that $P_W(y)=\sum \limits_{W \subseteq S \subseteq N} Q_S(y)$. Moreover, from the Burning algorithm (Lemma \ref{CFG:Dah}) it follows that $\set{c\in \mathcal{C}:\support{c}=\emptyset}=\emptyset$, therefore $Q_{\emptyset}(y)=0$. Applying the Mobius inversion formula for the Boolean lattice $2^{N}$ we have $0=Q_{\emptyset}(y)=\underset{W\subseteq N}{\sum} (-1)^{|W|} P_W(y)$, which allows to express $P_\emptyset(y)$ in terms of the other components of the sum,
$$\Tut{G}{y}=P_{\emptyset}(y)=\sum_{\substack{W\subseteq N\\ W \neq \emptyset}} (-1)^{|W|+1} P_{W}(y).$$

For the second part of the proof, we claim that $$P_W(y)=y^{\totalchipmin{H}-\totalchipmin{G}-\numarcs{G}{s}{W}}\frac{1}{(1-y)^{|W|}}\underset{v\in W}{\prod}\left(1-y^{\numarcs{G}{s}{v}}\right)\Tut{H}{y}$$ where $H$ denotes $\vertexcontract{G}{W\cup \set{s}}$.

The vertices in $W$ are denoted by $w_1,w_2,\cdots,w_p$ for some $p$, and let $\mathcal{C}'$ be the set of all recurrent configurations of $H$ with sink $s'$, where $s'$ is the new vertex in $H$ resulting from replacing the vertices in $W\cup \set{s}$. It follows from Lemma \ref{recurrent:contraction} and the definition of level that
$$
P_W(y)
=\underset{\substack{c\in \mathcal{C}\\W \subseteq \support{c}}}{\sum}y^{\level{G}{c}}
=\underset{c\in \mathcal{C'}}{\sum}\sum_{\substack{d\in \mathcal{C}\\d_{|V\backslash(W\cup \set{s})}=c}} y^{\level{G}{d}}
=\underset{c\in \mathcal{C}'}{\sum}\left( y^{-\totalchipmin{G}+\outdegree{G}{s}} ~\Big(\!\!\!\! \sum_{\substack{d\in \mathcal{C}\\d_{|V\backslash(W\cup \set{s})}=c}}\!\!\! y^{ \underset{v \in W}{\sum} d(v)} \Big)~ y^{ \underset{v\not \in (W\cup \set{s})}{\sum} \!\!\!\!c(v)}\right)
$$
For each $i\in [1..p]$, let $I_i:=\set{\outdegree{G}{s}-\numarcs{G}{s}{w_i},\outdegree{G}{s}-\numarcs{G}{s}{w_i}+1,\cdots,\outdegree{G}{s}-1}$. It follows from Lemma \ref{recurrent:contraction} that the map $\mu:I_1\times I_2\times\cdots\times I_p\times \mathcal{C}'\to \mathcal{C}$, defined by $\mu(i_1,i_2,\cdots,i_p,c)(v)$ is equal to $c(v)$ if $v\not \in W$, and equal to $i_j$ if $v=w_j$, is bijective, which means that, for a configuration $c \in \mathcal C'$, the configurations on the graph $G$ constructed from $c$ by putting any number of chips in $I_i$ to $w_i$ produces the whole set $\mathcal C$. As a consequence,
\begin{align*}
P_W(y)
&=\underset{c\in \mathcal{C}'}{\sum}\left( y^{-\totalchipmin{G}+\outdegree{G}{s}} ~\Big( \underset{1\leq i \leq p}{\prod} \underset{j\in I_i}{\sum}y^{j} \Big)~ y^{ \underset{v\not \in (W\cup \set{s})}{\sum} \!\!\!\!c(v)}\right)\\
&=y^{-\totalchipmin{G}+\outdegree{G}{s}} ~\Big( \underset{1\leq i \leq p}{\prod} \underset{j\in I_i}{\sum}y^{j} \Big)~ \underset{c\in \mathcal{C}'}{\sum} y^{ \underset{v\not \in (W\cup \set{s})}{\sum} \!\!\!\!c(v)}\\
&= y^{-\totalchipmin{G}+\outdegree{G}{s}} \underset{w\in W}{\prod} y^{\outdegree{G}{w}-\numarcs{G}{s}{w}}\underset{w\in W}{\prod} \frac{\left(1-y^{\numarcs{G}{s}{w}}\right)}{1-y} \underset{c\in \mathcal{C}'}{\sum}y^{\underset{v\not \in (W\cup \set{s})}{\sum} \!\!\!\!c(v)}\\
&=y^{-\totalchipmin{G}-\numarcs{G}{s}{W}}y^{\underset{v\in W\cup \set{s}}{\sum} \outdegree{G}{v}}\frac{1}{(1-y)^{|W|}}\underset{w\in W}{\prod} \left(1-y^{\numarcs{G}{s}{w}}\right)\underset{c\in \mathcal{C}'}{\sum} y^{\underset{v\not \in W\cup \set{s}}{\sum} \!\!\!\!c(v)}\\
&=y^{-\totalchipmin{G}-\numarcs{G}{s}{W}} y^{\outdegree{H}{s'}}\frac{1}{(1-y)^{|W|}}\underset{w\in W}{\prod}\left(1-y^{\numarcs{G}{s}{w}}\right)\underset{c\in \mathcal{C}'}{\sum} y^{\underset{v\neq W\cup \set{s}}{\sum} \!\!\!\!c(v)}\\
&=y^{\totalchipmin{H}-\totalchipmin{G}-\numarcs{G}{s}{W}}\frac{1}{(1-y)^{|W|}}\underset{w\in W}{\prod}\left(1-y^{\numarcs{G}{s}{w}}\right)\underset{c\in \mathcal{C}'}{\sum} y^{\outdegree{H}{s'}+\underset{v\not \in W\cup \set{s}}{\sum} \!\!\!\!c(v)-\totalchipmin{H}}\\
&=y^{\totalchipmin{H}-\totalchipmin{G}-\numarcs{G}{s}{W}}\frac{1}{(1-y)^{|W|}} \underset{w\in W}{\prod}\left(1-y^{\numarcs{G}{s}{w}}\right) \underset{c\in \mathcal{C}'}{\sum} y^{\level{H}{c}}\\
&=y^{\totalchipmin{H}-\totalchipmin{G}-\numarcs{G}{s}{W}} \frac{1}{(1-y)^{|W|}}\underset{w\in W}{\prod}\left(1-y^{\numarcs{G}{s}{w}}\right)\Tut{H}{y}
\end{align*}
which proves our claim and concludes the proof.
\end{proof}

\section{Some open problems}
\label{open section}

In this paper we defined a natural analogue of the Tutte polynomial in one variable, for the class of general Eulerian digraphs. From a sink-independence property of the generating function of the set of recurrent configurations of the Chip-firing game, it turns out that this polynomial $\Tut{G}{y}$ is characteristic of the support graph itself, regardless of the chosen sink. Most interestingly, this polynomial is equal to the well-known Tutte polynomial $T_G(1,y)$ on undirected graphs. We presented evaluations of $\Tut{G}{y}$ generalizing the evaluations of $T_G(1,y)$, and we hope that new objects counted by evaluations of $\Tut{G}{y}$ will be discovered. Finally, we showed recursive formulas for this polynomial, which again account for natural generalization of those of the Tutte polynomial on undirected graphs. We end up with a new recursive formula for $\Tut{G}{y}$ in order to get a complete set of recursive formulas defining this polynomial.

It is now natural to ask whether there exists such a natural generalization of $T_G(1,y)$ to the class of connected digraphs. We believe there is such a generalization to the class of strongly connected digraphs by the following surprising conjecture.


Let $G=(V,E)$ be a strongly connected digraph and $s$ be a vertex of $G$. We denote by $\deleteoutarcs{G}{s}$ the digraph constructed from $G$ by removing all out-going arcs of $s$. Clearly, $\deleteoutarcs{G}{s}$ has a global sink $s$. Fix a linear order $v_1\prec v_2\prec\cdots \prec v_{n-1}$ on the set of all vertices of $G$ distinct from $s$, where $n=|V|$. Let $r_1,r_2,\cdots,r_{n-1}\in \mathbb{Z}^{n-1}$ be given by $r_{i,j}=\numarcs{G}{v_i}{v_j}$ if $i\neq j$, and $r_{i,i}=\outdegree{G}{v_i}$, and let $\beta=(\beta_1,\beta_2,\cdots,\beta_{n-1})\in \mathbb{Z}^{n-1}$ be given by $\beta_i=\numarcs{G}{s}{v_i}$. We define a equivalence relation $\sim$ on the set $\mathcal{C}$ of all recurrent configurations of $\deleteoutarcs{G}{s}$ by $c_1\sim c_2$ iff $c_1-c_2\in <r_1,r_2,\cdots,r_{n-1},\beta>$, where $<r_1,r_2,\cdots,r_{n-1},\beta>$ is the subgroup of $(\mathbb{Z}^{n-1},+)$ generated by $r_1,r_2,\cdots,r_{n-1},\beta$. Note that if $G$ is Eulerian then $\beta\in <r_1,r_2,\cdots,r_{n-1}>$, therefore $<r_1,r_2,\cdots,r_{n-1},\beta>=<r_1,r_2,\cdots,r_{n-1}>$. For each $B\in \mathcal{C}/\!\!\sim$ let $\confdegree{G}{s}{B}$ denote $\max\set{\outdegree{G}{s}+\underset{v\neq s}{\sum}c(v):c\in B}$.

\begin{conj}
The sequence $\left(\confdegree{G}{s}{B}\right)_{B \,\in\, \mathcal{C}/\sim}$ is independent of the choice of $s$, up to a permutation on the entries.
\end{conj}

\noindent If the conjecture holds, we have a generalization of $T_G(1,y)$ to the class of strongly connected digraphs.

It may be more reasonable to ask whether there is a generalization of the Tutte polynomial in two variables to the class of Eulerian digraphs. The bijection presented in \cite{CB03} gives a promising direction for this problem, that is, to look for its generalization to the class of Eulerian digraphs. In addition, one has to generalize the concepts of internal and external activities to the class of Eulerian digraphs. This task is hard, but the generalization of bridge presented in this paper may give insights to address the question.

This question could be addressed by looking for an alternative recursive formula for the Tutte polynomial in two variables on undirected graphs so that it works on Eulerian digraphs, possibly for general digraphs. The new recursive formula in Proposition \ref{general recursive formula} could suggest such a formula since it uses only the vertex contraction in its recursive terms, and the notion of vertex contraction has a natural generalization to general digraphs. Moreover, the following conjecture convinces that such a generalization exists.
\begin{conj}
Let $G$ be a connected undirected graph, $s$ a vertex of $G$, and $N$ the set of all neighbors of $s$ (not including $s$). Then $T_G(x,y)$ is in the ideal generated by $\set{T_{\overline{\vertexcontract{G}{W\cup \set{s}}}}(x,y):\emptyset\subsetneq W\subseteq N}$ in $\mathbb{Q}[x,y]$, where $\overline{H}$ denotes $H$ in which all loops have been removed.
\end{conj}

Equivalently, the conjecture means that there exist polynomials $P_{W}(x,y)\in \mathbb{Q}[x,y],\emptyset\subsetneq W \subseteq N$ such that $T_G(x,y)=\underset{\emptyset \subsetneq W\subseteq N}{\sum} P_{W}(x,y) T_{\overline{G_{/W\cup \set{s}}}}(x,y)$.  Let us give an example for the illustrative explanation of this conjecture. The first graph in Figure \ref{fig:im24} shows an undirected graph $G$ with a chosen vertex $s$ (in black). The remaining graphs are the graphs which are obtained from $G$ by contracting vertex sets $\set{s}\cup W$ and then removing the resulting loops. The Tutte polynomials are shown below the corresponding graphs. By using a  Gr\"{o}bner basis we can verify that the first polynomial is in the ideal generated by the remaining polynomials.
\begin{figure}
\centering
\includegraphics[bb=0 0 759 759,width=5.36in,height=4.19in]{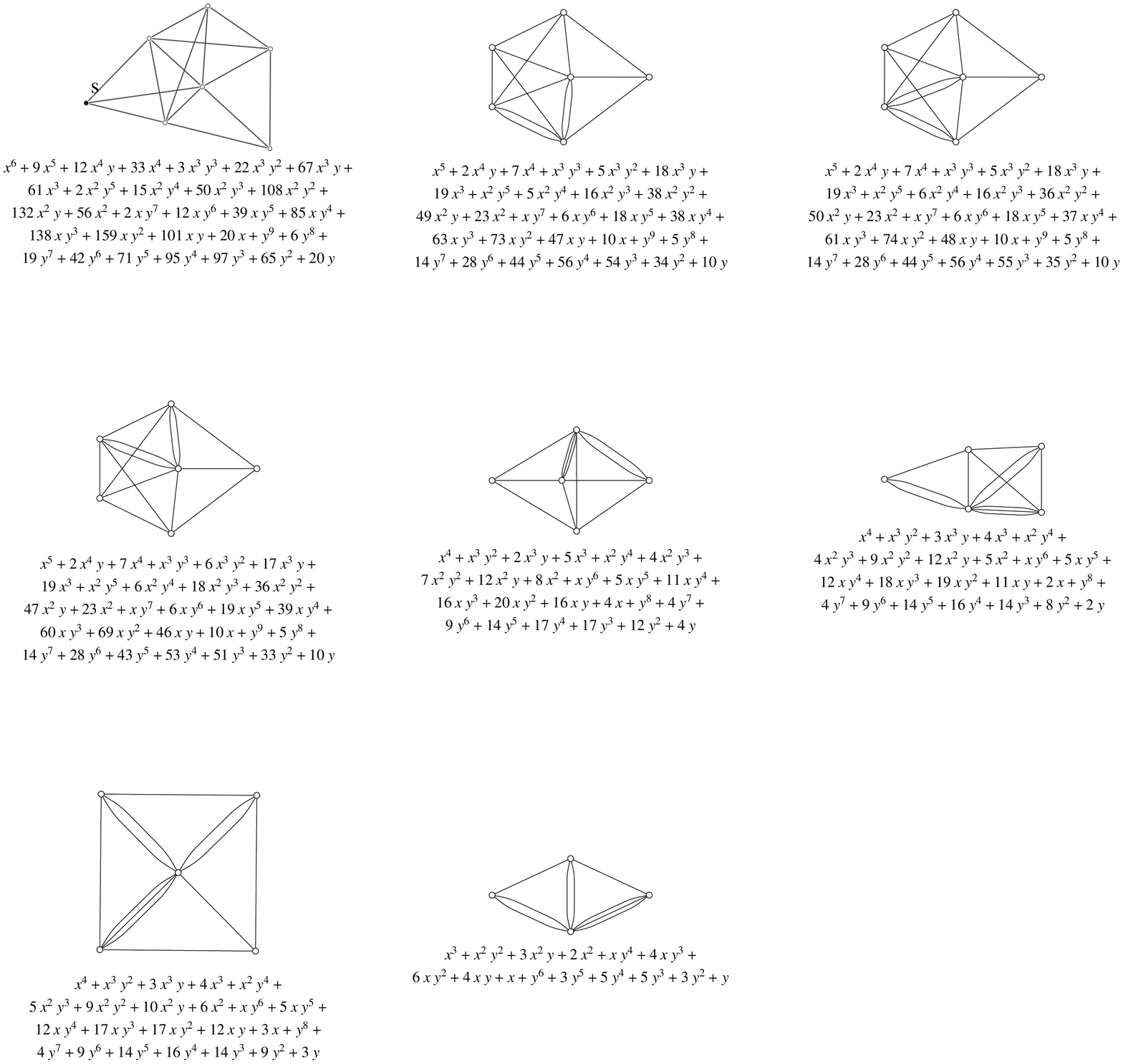}
\caption{An undirected graph and its vertex contractions}
\label{fig:im24}
\end{figure}
\pagebreak

K\'evin Perrot\\
Universit\'e de Lyon - LIP (UMR 5668 CNRS-ENS de Lyon-Universit\'e Lyon 1)\\
46 all\'ee d'Italie 69364 Lyon Cedex 7-France\\
Universit\'e de Nice Sophia Antipolis - Laboratoire I3S (UMR 6070 CNRS)\\
2000 route des Lucioles, BP 121, F-06903 Sophia Antipolis Cedex, France\\
Email: kevin.perrot@ens-lyon.fr\\

\noindent Trung Van Pham\\
Vietnamese Institute of Mathematics\\
18 Hoang Quoc Viet Road, Cau Giay District, Hanoi, Viet Nam\\
Email: pvtrung@math.ac.vn
\end{document}